\pdfoutput=1
\RequirePackage{ifpdf}
\ifpdf 
\documentclass[pdftex]{sigma}
\else
\documentclass{sigma}
\fi

\numberwithin{equation}{section}

\newtheorem{Theorem}{Theorem}[section]
\newtheorem*{Theorem*}{Theorem}
\newtheorem{Corollary}[Theorem]{Corollary}

\newtheorem{Proposition}[Theorem]{Proposition}
 { \theoremstyle{definition}
\newtheorem{Definition}[Theorem]{Definition}

 }

\begin{document}
\allowdisplaybreaks

\newcommand{\arXivNumber}{2201.03960}

\renewcommand{\PaperNumber}{056}

\FirstPageHeading

\ShortArticleName{$q$-Middle Convolution and $q$-Painlev\'e Equation}

\ArticleName{$\boldsymbol{q}$-Middle Convolution and $\boldsymbol{q}$-Painlev\'e Equation}

\Author{Shoko SASAKI~$^{\rm a}$, Shun TAKAGI~$^{\rm a}$ and Kouichi TAKEMURA~$^{\rm b}$}

\AuthorNameForHeading{S.~Sasaki, S.~Takagi and K.~Takemura}

\Address{$^{\rm a)}$~Department of Mathematics, Faculty of Science and Engineering, Chuo University,\\
\hphantom{$^{\rm a)}$}~1-13-27 Kasuga, Bunkyo-ku, Tokyo 112-8551, Japan}

\Address{$^{\rm b)}$~Department of Mathematics, Ochanomizu University,\\
\hphantom{$^{\rm b)}$}~2-1-1 Otsuka, Bunkyo-ku, Tokyo 112-8610, Japan}
\EmailD{\href{mailto:takemura.kouichi@ocha.ac.jp}{takemura.kouichi@ocha.ac.jp}}

\ArticleDates{Received January 31, 2022, in final form July 08, 2022; Published online July 20, 2022}

\Abstract{A $q$-deformation of the middle convolution was introduced by Sakai and Yama\-guchi. We apply it to a linear $q$-difference equation associated with the $q$-Painlev\'e~VI equation. Then we obtain integral transformations. We investigate the $q$-middle convolution in terms of the affine Weyl group symmetry of the $q$-Painlev\'e VI equation. We deduce an integral transformation on the $q$-Heun equation.}


\Classification{33E10; 34M55; 39A13}

\section{Introduction}

The middle convolution was introduced by Katz \cite{Katz} for local systems on a punctured Riemann sphere, and Dettweiler and Reiter \cite{DR1,DR2} reformulated it for the Fuchsian system of differential equations.
Here the Fuchsian system of differential equations is the system of linear differential equations written as
\begin{gather}
\frac{{\rm d}Y}{{\rm d}x}=\bigg(\frac{A_1}{x-t_1}+\frac{A_2}{x-t_2}+\dots + \frac{A_r}{x-t_r}\bigg) Y, \label{eq:original}
\end{gather}
where $Y$ is a column vector with $n$ entries and $A_1, A_2, \dots,A_r$ are constant matrices of size $n\times n$.
We review briefly the definition of the middle convolution for equation~(\ref{eq:original}) (or the tuple of the matrices $(A_1, \dots,A_r)$).
Let $\lambda \in \mathbb{C}$ and $F_i$, $i=1,\dots,r$, be the matrix of size $nr \times nr $ of the form
\begin{gather}
F_{i} =
\begin{pmatrix}
O & \cdots & O & \cdots & O \\
\vdots & & \vdots & & \vdots \\
A_{1} & \cdots & A_{i} + \lambda I_{n} & \cdots & A_{r} \\
\vdots & & \vdots & & \vdots \\
O & \cdots & O & \cdots & O
\end{pmatrix} {\scriptstyle (i)},
\label{eq:Fi}
\end{gather}
where $I_n $ is the identity matrix of size $n$.
Then the correspondence of the tuple of matrices $(A_1, \dots,A_r ) \mapsto (F_1, \dots,F_r )$ (or the correspondence of the associated Fuchsian system) is called the convolution.
The convolution does not preserve the irreducibility in general.
It is shown that the following subspaces $\mathcal{K}, \mathcal{L}$ of $\mathbb{C}^{nr}$ are preserved by the action of $F_i$, $i=1,\dots,r$,
\begin{gather*}
\mathcal{K} = \begin{pmatrix}
\ker A_{1} \\
\vdots \\
\ker A_{r}
\end{pmatrix}\!, \qquad
\mathcal{L} = \ker (F_{1} + F_{2} + \cdots + F_{r}).
\end{gather*}
We denote the linear transformation induced from the action of $F_{i}$ on the quotient space $\mathbb{C}^{nr}/(\mathcal{K} + \mathcal{L})$ by $\overline{F}_{i}$.
The correspondence of the tuple of matrices $(A_1, \dots,A_r ) \mapsto \big(\overline{F}_1, \dots,\overline{F}_r \big)$ (or~the correspondence of the associated Fuchsian system) is called the middle convolution.
It~was shown in \cite{DR1} that the convolution is related with Euler's integral transformation.
Let~$Y(x)$ be a solution of equation~(\ref{eq:original}).
Set
\begin{gather*}
W_{j}(x) = \frac{Y(x)}{x - t_{j}}, \qquad
W(x) = \begin{pmatrix}
W_{1}(x) \\
\vdots \\
W_{r}(x)
\end{pmatrix}\!.
\end{gather*}
Then $W(x)$ is a column vector with $nr$ entries.
We apply Euler's integral transformation for each entry of $W(x)$, i.e., we set
\begin{gather*}
G(x) = \int_{\Delta}W(s)(x - s)^{\lambda}\,{\rm d}s,
\end{gather*}
where $\Delta $ is an appropriate cycle in $\mathbb{C} $ with the variable~$s$.
Then the function $G(x)$ satisfies the following Fuchsian system of differential equation
\begin{gather*}
\frac{{\rm d}Y}{{\rm d}x} = \bigg( \frac{F_1}{x-t_1}+\frac{F_2}{x-t_2}+\dots + \frac{F_r}{x-t_r} \bigg)Y,
\end{gather*}
where $F_1, \dots, F_r$ were defined in equation~(\ref{eq:Fi}).

Sakai and Yamaguchi \cite{SY} constructed a theory of a $q$-deformation of the middle convolution for systems of $q$-difference equations.
Here the system is described as
\begin{gather*}
Y(qx) = \bigg( B_{\infty} + \frac{B_1}{1 - x/ t_1} +\dots + \frac{B_r}{1 - x/t_r } \bigg) Y(x),
\end{gather*}
where $Y(x)$ is a column vector with $n$ entries and $B_{\infty}, B_1, \dots,B_r$ are constant matrices of size $n \times n$.
The construction of the $q$-middle convolution is similar to the case of the Fuchsian system of differential equations.
For details, see Section~\ref{sec:qMC}.

In this paper we apply the $q$-middle convolution to linear $q$-difference equations which are related to the $q$-Painlev\'e VI equation
\begin{gather}
\frac{y\overline{y}}{a_{3}a_{4}} =
\frac{(\overline{z} - tb_{1})(\overline{z} - tb_{2})}{(\overline{z} - b_{3})(\overline{z} - b_{4})}, \qquad
\frac{z\overline{z}}{b_{3}b_{4}} =
\frac{(y - ta_{1})(y - ta_{2})}{(y - a_{3})(y - a_{4})},\label{eq:qPVI}
\end{gather}
with the constraint $b_1 b_2 a_3 a_4 = q a_1 a_2 b_3 b_4$.
Here $\overline{y}$ and $\overline{z}$ denotes the time evolution $t \mapsto qt$ of $y$ and $z$, and the parameters $a_1,\dots, a_4, b_1,\dots,b_4$ are time-independent.
The $q$-Painlev\'e VI equation was introduced by Jimbo and Sakai \cite{JS} as a $q$-deformation of the Painlev\'e VI equation.
They obtained equation~(\ref{eq:qPVI}) by introducing a $q$-analogue of the monodromy preserving deformation, and it is related to the linear $q$-difference equations
\begin{align}
& Y(qx )= A(x ) Y(x ), \label{eq:JSlinDEqx}
\end{align}
where $A(x)$ is a $2\times 2 $ matrix with polynomial entries (see equation~(\ref{eq:Ax012}) for details).
The linear $q$-difference equation which we apply the $q$-middle convolution is not equation~(\ref{eq:JSlinDEqx}) but the transformed equation
\begin{gather}
Y(qx )= B(x ) Y(x ), \nonumber
\\
 B(x) = \frac{A(x)}{c_0(x-ta_1)(x-ta_2)} = B_{\infty} + \frac{B_1}{1 - x/(ta_1)} + \frac{B_2}{1 - x/(ta_2)}
 \label{eq:JSlinDEqxBintro}
\end{gather}
for some constant $c_0$.
By applying the $q$-middle convolution with the parameter $\lambda $, we obtain $2\times 2$ matrices, if we choose the constants $c_0$ and $\lambda $ suitably.
Then we obtain a correspondence of the parameters, and we may regard it as a correspondence of the $q$-Painlev\'e VI equations.
On the other hand, it is known that the $q$-Painlev\'e VI equation has a symmetry of the affine Weyl group of type $D^{(1)}_5$, and a realization of the symmetry was described in the review \cite{KNY} of Kajiwara, Noumi and Yamada.
In this paper, we express the symmetry by the $q$-middle convolution in terms of the generators of the affine Weyl group.

In \cite{STT1}, a relationship between the $q$-Painlev\'e VI equation and the $q$-Heun equation
\begin{gather*}
\big(x-h_{1}q^{1/2}\big)\big(x-h_{2}q^{1/2}\big)g(x/q)+l_{3}l_{4}\big(x-l_{1}q^{-1/2}\big)\big(x-l_{2}q^{-1/2}\big)g(qx)
\\ \qquad
{}-\big\{(l_{3}+l_{4})x^{2}+Ex+(l_{1}l_{2}l_{3}l_{4}h_{1}h_{2})^{1/2}\big(h_{3}^{1/2}+h_{3}^{-1/2}\big)\big\}g(x)=0
\end{gather*}
was studied from a viewpoint of the initial value space.
In particular, the $q$-Heun equation was obtained from the linear $q$-difference equation associated to the $q$-Painlev\'e VI equation by specializing the parameters.
On the other hand, the $q$-middle convolution induces an integral transformation of the linear $q$-difference equation.
By considering a particular specialization, the linear $q$-difference equation turns out to be the $q$-Heun equation and we obtain an integral transformation of $q$-Heun equation.

This paper is organized as follows.
In Section~\ref{sec:qMC}, we review a part of the theory of the $q$-middle convolution established by Sakai and Yamaguchi~\cite{SY}.
In Section~\ref{sec:JSqMC}, we recall the linear $q$-difference equation associated to the $q$-Painlev\'e VI equation and calculate the $q$-middle convolution for it.
In Section~\ref{sec:qmcWeyl}, we investigate the symmetry by the $q$-middle convolution in terms of the Weyl group symmetry of the $q$-Painlev\'e VI equation.
For this purpose, we clarify a~relationship between equation~(\ref{eq:JSlinDEqxBintro}) and the Lax pair in~\cite{KNY}.
In Section~\ref{sec:ITHeun}, we obtain an integral transformation on the $q$-Heun equation.
In Section~\ref{sec:CR}, we give concluding remarks.

\section[q-middle convolution]
{$\boldsymbol q$-middle convolution} \label{sec:qMC}
We recall the $q$-convolution and the $q$-middle convolution introduced by Sakai and Yamagu\-chi~\cite{SY}.

Let $\mathbf{B}= ( B_{\infty} ; B_{1},\dots,B_N ) $ be the tuple of the square matrices of the same size and $\mathbf{b}= (b_1, b_2, \dots,b_N)$ be the tuple of the non-zero complex numbers which are different from one another.
We denote by $E_{\mathbf{B}, \mathbf{b}}$ the linear $q$-difference equations
\begin{gather*}
Y(q x) = B(x)Y(x), \qquad B(x) = B_{\infty} + \sum^{N}_{i = 1}\frac{B_{i}}{1 - x/b_{i}}.
\end{gather*}

\begin{Definition}[$q$-convolution, \cite{SY}] \label{def:qc}
Let $\mathbf{B}= ( B_{\infty}; B_{1},\dots,B_N ) $ be the tuple of $m\times m $ matrices and $(b_1, b_2, \dots,b_N)$ be the tuple of the non-zero complex numbers which are different one another.
Set $B_0 = I_m - B_\infty - B_{1} - \dots -B_N$,
We define the $q$-convolution $c_\lambda\colon ( B_{\infty};B_{1},\dots,B_N ) \mapsto ( F_{\infty};F_{1},\dots,F_N )$ as follows:
\begin{gather*}
\mbox{\boldmath $F$} = ( F_\infty ; F_1, \dots, F_N) \mbox{ \rm is a tuple of $(N+1)m \times (N+1)m$ matrices,}
\\
F_i = \begin{pmatrix}
 {} & {} & O & {} & {} \\[5pt]
 B_0 & \cdots & B_i - \big(1-q^\lambda\big)I_m & \cdots & B_N \\[5pt]
 {} & {} & O & {} & {}
 \end{pmatrix}
{\scriptstyle(i+1)}, \qquad 1\leq i \leq N, \\[5pt]
 F_\infty = I_{(N+1)m} - \widehat{F}, \qquad
 \widehat{F} = \begin{pmatrix}
 B_0 & \cdots & B_N \\[5pt]
 \vdots & \ddots & \vdots \\[5pt]
 B_0 & \cdots & B_N
 \end{pmatrix}\!.
\end{gather*}
\end{Definition}

Let $\xi \in {\mathbb {C}} \setminus \{ 0 \}$.
The $q$-convolution induces the $q$-analogue of the Euler's integral transformation in terms of the Jackson integral
\begin{gather*}
 \int^{\xi \infty}_{0}f(x) \, {\rm d}_{q}x = (1-q)\sum^{\infty}_{n=-\infty}q^{n} \xi f(q^n \xi )
\end{gather*}
for the solutions of the $q$-difference equations.
Note that the value of the Jackson integral may depend on the value $\xi $.

\begin{Theorem}[{\cite[Theorem 2.1]{SY}}]\label{thm:qcint}
Let $Y(x)$ be a solution of $E_{\mathbf{B}, \mathbf{b}}$.
Set $b_0 =0$ and
\begin{align*}
& P_{\lambda}(x, s) = \frac{\big(q^{\lambda +1} s/x;q\big)_{\infty }}{(q s/x;q)_{\infty }} = \prod^{\infty}_{i = 0}\frac{x - q^{i + \lambda + 1}s}{x - q^{i + 1}s}.
\end{align*}
Define the function $\widehat{Y}(x) $ by
\begin{gather*}
\widehat{Y}_{i}(x) = \int^{\xi \infty}_{0}\frac{P_{\lambda}(x, s)}{s-b_{i}}Y(s) \, {\rm d}_{q}s,\quad i=0,\dots,N,\qquad
\widehat{Y}(x) =\begin{pmatrix} \widehat{Y}_{0}(x) \\ \vdots \\ \widehat{Y}_{N}(x) \end{pmatrix}\!.
\end{gather*}
Then the function $\widehat{Y}(x)$ satisfies the equation $E_{\mathbf{F}, \mathbf{b}}$, i.e.,
\begin{gather*}
\widehat{Y}(q x) = \bigg( F_{\infty} + \sum^{N}_{i = 1}\frac{F_{i}}{1 - x/b_{i}} \bigg) \widehat{Y}(x).
\end{gather*}
\end{Theorem}

Although the original theorem by Sakai and Yamaguchi was restricted to the case $\xi =1$ in the Jackson integral, we may just extend it to the case $\xi \in {\mathbb {C}} \setminus \{ 0 \}$, which was motivated by the theory of the Jackson integral due to Aomoto~\cite{Aom}.
The convergence of the Jackson integrals~$\widehat{Y}_{i}(x)$, $i=0,\dots,N$, would not be considered in~\cite{SY}.
In this paper, we discuss the Jackson integrals formally and we do not consider the convergence in details.
Namely, we discuss the Jackson integrals under the assumption that the integrals converge absolutely and Theorem \ref{thm:qcint} holds true with the convergence.
Thus we use the phrasing ``formally'' in theorems which are related to the Jackson integral.
An aspect of convergence on Theorem \ref{thm:qcint} will be discussed in~\cite{AT}.

The $q$-middle convolution is defined by considering an appropriate quotient space.

\begin{Definition}[$q$-middle convolution, \cite{SY}]
We define the $\boldsymbol{F}$-invariant subspaces $\mathcal{K}$ and $\mathcal{L}$ of~$(\mathbb{C}^m )^{N+1}$ as follows
\begin{gather*}
 \mathcal{K} = \mathcal{K}_\mathcal{V} =
\begin{pmatrix} \ker B_0 \\ \vdots \\ \ker B_N \end{pmatrix}\!, \qquad
 \mathcal{L} = \mathcal{L}_{\mathcal{V}}(\lambda) =
 \ker \big(\widehat{F} - \big(1 - q^{\lambda}\big)I_{(N+1)m}\big).
\end{gather*}
We denote the action of $F_k$ on the quotient space $(\mathbb{C}^m )^{N+1}/(\mathcal{K} + \mathcal{L})$ by $\overline{F}_k$, $k=\infty,1, \dots,N$.
Then the $q$-middle convolution $mc_\lambda$ is defined by the correspondence $ E_{\mathbf{B}, \mathbf{b}}\mapsto E_{\overline{\mathbf{F}}, \mathbf{b}} $, where $ \overline{\mathbf{F}} = \big( \overline{F}_{\infty}; \overline{F}_{1},\dots,\overline{F}_N \big) $.
\end{Definition}
The $q$-middle convolution $mc_\lambda$ would induce the integral transformation of the solutions by applying the integral transformation on the $q$-convolution, although it would be necessary to consider the subspace $\mathcal{K} + \mathcal{L} \subset (\mathbb{C}^m )^{N+1}$.

\section[Linear q-difference equation associated to q-Painleve VI equation and q-middle convolution]
{Linear $\boldsymbol q$-difference equation associated \\ to $\boldsymbol q$-Painlev\'e VI equation and $\boldsymbol q$-middle convolution} \label{sec:JSqMC}

We recall the linear $q$-difference equation
\begin{align}
& Y(qx )= A(x ) Y(x ), \label{eq:JSlinDEqx1}
\end{align}
which was discussed by Jimbo and Sakai \cite{JS} to obtain the $q$-Painlev\'e VI equation by the connection preserving deformation.

We take the $2\times 2$ matrix $A(x )$ in equation~(\ref{eq:JSlinDEqx1}) to be of the form
\begin{gather}
A(x)= A_0 (t) +A_1 (t) x +A_2 x^2,\qquad
A_2 = \begin{pmatrix}
\chi _1 & 0 \\ 0 & \chi _2 \end{pmatrix}\!,\nonumber
\\
A_0(t) \mbox{ has eigenvalues } t\theta _1,\ t\theta _2,\nonumber
\\
\det A(x)= \chi _1 \chi _2 (x-ta_1)(x-ta_2)(x-a_3)(x-a_4). \label{eq:Ax012}
\end{gather}
Then we have the following relation
\begin{gather}
\chi_{1}\chi_{2} a_{1} a_{2} a_{3} a_{4} = \theta_{1}\theta_{2}. \label{eq:relation1}
\end{gather}
Note that the relations to the parameter of the $q$-Painlev\'e VI equation in equation~(\ref{eq:qPVI}) are given by $ b_1= a_1 a_2 /\theta _1$, $ b_2= a_1 a_2 /\theta _2$, $b_3 =1/(q\chi _1)$, $b_4 = 1/\chi _2$.

We need accessory parameters to determine uniquely the elements of the matrix $A(x)$.
Write
\begin{gather*}
A(x)=
\begin{pmatrix}
a_{11}(x) & a_{12}(x) \\
a_{21}(x) & a_{22}(x)
\end{pmatrix}\!.
\end{gather*}
Then $a_{12} (x)$ is a linear polynomial.
We introduce the parameters $w$, $y$, $z$ and impose the condition
\begin{gather}
a_{12} (x) = \chi_2 w (x-y),\qquad
a_{11} (x) | _{x=y } = (y - t a_{1})(y -t a_{2}) /(q z). \label{eq:wyz}
\end{gather}
Then the elements of $A(x)$ are determined as
\begin{gather}
	A(x) = \begin{pmatrix}
	 \chi_1((x-y)(x-\alpha)+z_1) & \chi_2w(x-y) \\
	 \chi_1w^{-1}(\gamma x + \delta) & \chi_2((x-y)(x-\beta)+z_2)
	 \end{pmatrix}\!, \label{eq:JSlaxpair}
\end{gather}
where
\begin{gather*}
\alpha = \frac{1}{\chi_1 - \chi_2}\big[y^{-1}((\theta_1 + \theta_2)t
-\chi_1z_1 - \chi_2z_2) - \chi_2((a_1 + a_2)t + a_3+a_4 - 2y)\big],
\\
\beta = \frac{1}{\chi_1 - \chi_2}\big[{-}y^{-1}((\theta_1 + \theta_2)t
-\chi_1z_1 - \chi_2z_2) + \chi_1((a_1 + a_2)t + a_3+a_4 - 2y)\big],
\\
\gamma = z_1+z_2+(y+\alpha)(y+\beta)+(\alpha+\beta)y-a_1a_2t^2-
(a_1+a_2)(a_3+a_4)t -a_3a_4 \nonumber,
\\
\delta = y^{-1}\big(a_1a_2a_3a_4t^2 - (\alpha y+z_1)(\beta y+z_2)\big)
\end{gather*}
and
\begin{gather*}
z_{1} = \frac{(y - ta_{1})(y - ta_{2})}{q\chi_{1}z}, \qquad
z_{2} = q\chi_{1}(y - a_{3})(y - a_{4})z.
\end{gather*}

We consider the $q$-middle convolution for the $q$-difference equation
\begin{align}
& Y(qx )= B(x ) Y(x ),\qquad B(x) = \frac{A(x)}{c_0(x-ta_1)(x-ta_2)}, \label{eq:JSlinDEqxB}
\end{align}
where $c_0$ is a constant which will be fixed later.
Note that, if $\tilde{Y} (x)$ is a solution of equation~(\ref{eq:JSlinDEqx1}) and the parameter $\mu $ satisfies $ q^{\mu } =1/\big(c_0 a_1 a_2 t^2\big)$, then the function
\begin{gather*}
Y(x) = x^{\mu } (x/(ta_1);q)_{\infty } (x/(ta_2);q)_{\infty } \tilde{Y} (x)
\end{gather*}
satisfies equation~(\ref{eq:JSlinDEqxB}).
Write
\begin{gather}
 B(x) = B_{\infty} + \frac{B_1}{1 - x/(ta_1)}
 + \frac{B_2}{1 - x /(ta_2 )} = \begin{pmatrix}
		b_{11}(x) & b_{12}(x) \\
		b_{21}(x) & b_{22}(x)
	\end{pmatrix}\!. \label{eq:Bx}
\end{gather}
Then
\begin{gather}
B_{\infty} = \frac{1}{c_0} \begin{pmatrix}
 \chi _1 & 0 \\
 0 & \chi _2
 \end{pmatrix}\!, \qquad
 B_{2} = B_{1} | _{a_1 \leftrightarrow a_2}, \label{eq:BinftyB2B1}
 \\
B_{1} = \frac{a_2}{t \theta_1 (a_1-a_2)} \begin{pmatrix}
 (y-t a_1 ) b_1^{[1]} / (q y z (\chi _1-\chi _2)) & w \chi _2 (y-t a_1 ) \\[1ex]
 - b_1^{[1]} b_2^{[1]} / \big( q^2 w y^2 z^2 (\chi _1-\chi _2)^2 \big) & -\chi _2 b_2^{[1]} /(q y z (\chi _1-\chi _2) )
 \end{pmatrix}\!, \nonumber
\end{gather}
where
\begin{align*}
b_1^{[1]} &= q^2 \chi _1^2 \chi _2 (y-a_3) (y-a_4) z^2 + (y-t a_2 ) (\chi _2 y- \chi _1 t a_1 )
\\
&\phantom{=} -q \chi _1 \big\{ 2 \chi _2 y^2 - (\chi _1 t a_1 + \chi _2 t a_2 + \chi _2 a_3 + \chi _2 a_4 ) y +t (\theta_1 + \theta_2 ) \big\} z,
\\
b_2^{[1]} &= q^2 \chi _1 (y-a_3) (y-a_4) (\chi _1 y- \chi _2 t a_1 ) z^2 + (y-t a_1 )^2 (y- t a_2 )
\\
&\phantom{=} -q (y- t a_1 ) \big\{ 2 \chi _1 y^2 -(\chi _2 t a_1 + \chi _1 t a_2 + \chi _1 a_3 + \chi _1 a_4) y + t (\theta_1 + \theta_2 ) \big\} z.
\end{align*}
It is shown directly that $\det B_1=0$ and $\det B_2 =0$.
Set $B_0 = I_2 - B_{\infty} - B_1 - B_2 (= I_2 - B(0))$.
Then the condition $\det B_0=0$ is equivalent to $c_0 = \theta_1 /(ta_1a_2 ) $ or $c_0 = \theta_2 /(ta_1a_2 ) $.
We now impose the condition $\det B_0=0 $.
For this purpose, we restrict to the case $c_0 = \theta_1 /(t a_1a_2 ) $.
Note that the case $c_0= \theta_2 /(t a_1a_2 ) $ can be discussed by replacing the parameters as $\theta _1 \leftrightarrow \theta _2 $.
We eliminate the parameter $\theta_{2}$ by equation~(\ref{eq:relation1}).
It follows from $\det B_0=0$, $\det B_1=0$ and $\det B_2=0$ that there exists non-zero vectors
$\left(\begin{smallmatrix} v_{01} \\ v_{02} \end{smallmatrix} \right)$,
$\left(\begin{smallmatrix} v_{11} \\ v_{12} \end{smallmatrix} \right)$,
$\left(\begin{smallmatrix} v_{21} \\ v_{22} \end{smallmatrix} \right)$
 such that
\begin{gather}
 B_0\begin{pmatrix} v_{01} \\ v_{02} \end{pmatrix} =
 \begin{pmatrix} 0 \\ 0 \end{pmatrix}\!, \qquad
 B_1\begin{pmatrix} v_{11} \\ v_{12} \end{pmatrix} =
 \begin{pmatrix} 0 \\ 0 \end{pmatrix}\!, \qquad
 B_2\begin{pmatrix} v_{21} \\ v_{22} \end{pmatrix} =
 \begin{pmatrix} 0 \\ 0 \end{pmatrix}\!.
\label{eq:v01v02v11v12v21v22}
\end{gather}
We normalize the vectors by setting
\begin{gather}
v_{01} = q w y z \theta_1 (\chi_1 -\chi_2), \quad \
v_{11} = q w y z \theta_1 \chi_2 (\chi_1 -\chi_2), \quad \
v_{21} = q w y z \theta_1 \chi_2 (\chi_1 -\chi_2).
\label{eq:v01v11v21}
\end{gather}
We now apply Definition \ref{def:qc}.
Namely we set
\begin{gather}
 F_1 = \begin{pmatrix}
 {O} & O & {O} \\
 B_0 & B_1 - \big(1-q^\lambda\big)I_2 & B_2 \\
 {O} & O & {O}
 \end{pmatrix}\!, \qquad
 F_2 = \begin{pmatrix}
 {O} & O & {O} \\
 {O} & O & {O} \\
 B_0 & B_1 & B_2 - \big(1-q^\lambda\big)I_2
 \end{pmatrix}\!, \nonumber
 \\
 \widehat{F} = \begin{pmatrix}
 B_0 & B_1 & B_2 \\
 B_0 & B_1 & B_2 \\
 B_0 & B_1 & B_2
 \end{pmatrix}\!, \qquad
 F_{\infty} = I_6 - \widehat{F}.\label{eq:F1F2Finf}
\end{gather}
The invariant subspaces $\mathcal{K} $ and $\mathcal{L} $ are described as
\begin{gather}
 \mathcal{K} = \mathcal{K}_\mathcal{V} = \begin{pmatrix}
 \ker B_0 \\
 \ker B_1 \\
 \ker B_2
 \end{pmatrix}\!, \qquad
 \mathcal{L} = \mathcal{L}_{\mathcal{V}}(\lambda) =
 \ker \big(\widehat{F} - \big(1 - q^{\lambda}\big)I_6\big).
\label{eq:KL}
\end{gather}
Hence a basis of the space $\mathcal{K}$ is
\begin{gather*}
\left\{\!
\begin{pmatrix} v_{01} \\ v_{02} \\ 0 \\ 0 \\ 0 \\ 0 \end{pmatrix}\!,
\begin{pmatrix} 0 \\ 0 \\ v_{11} \\ v_{12} \\ 0 \\ 0 \end{pmatrix}\!,
\begin{pmatrix} 0 \\ 0 \\ 0 \\ 0 \\ v_{21} \\ v_{22} \\ \end{pmatrix}\!
 \right\}\!.
\end{gather*}
If $q^{\lambda}= \chi_{2} t a_1a_2 / \theta_1 $ (resp.~$q^{\lambda}= \chi_{1} t a_1a_2 / \theta_1 $) then $ \dim (\mathcal{L}) =1$ and the vector ${}^{t} (0,1,0,1,0,1)$ (resp.~${}^{t} (1,0,1,0,1,0)$) is a basis of the space $\mathcal{L} $.
Here we continue the discussion by setting~$q^{\lambda}= \chi_{2} t a_1a_2 / \theta_1 $.

We introduce the matrix $P$ by
\begin{gather}
P = \begin{pmatrix}
0 & 0 & 0 & v_{01} & 0 & 0 \\
g_{1} & g_{3} & 1 & v_{02} & 0 & 0 \\
0 & 0 & 0 & 0 & v_{11} & 0 \\
g_{2} & g_{4} & 1 & 0 & v_{12} & 0 \\
0 & 0 & 0 & 0 & 0 & v_{21} \\
 0 & 0 & 1 & 0 & 0 & v_{22} \\
\end{pmatrix}\!, \qquad
\begin{array}{l}
\displaystyle g_{1} = qz\theta_1 + y a_2 - q y z \chi_1 a_2 - t a_1a_2,\\[.5ex]
\displaystyle g_{2} = y ( q z\chi _1 -1 ) (a_1 - a_2),\\[.5ex]
\displaystyle g_{3} = -a_2(y - t a_1 ),\\[.5ex]
\displaystyle g_{4} = y (a_1 - a_2).
\end{array}
\label{eq:P}
\end{gather}
Then $\det P = -q^4 w^3 y^4 z^4 \theta_1^3 \chi _2^2 (\chi _1 -\chi _2)^3 (a_1-a_2) (\chi _1 t a_1 a_2 -\theta_1 )$, and the matrix $P$ is invertible if $\det P \neq 0 $.
Set
\begin{gather*}
\widetilde{F}_{1} = P^{-1}F_{1}P, \qquad
\widetilde{F}_{2} = P^{-1}F_{2}P, \qquad
\widetilde{F}_{\infty} = P^{-1}{F}_{\infty }P.
\end{gather*}
Then it follows from the invariance of the space $\mathcal{K} + \mathcal{L} $ that the $i.j$ elements of these matrices for $i \in \{ 1,2 \}$ and $j \in \{3,4,5,6 \}$ are equal to zero.
Thus, they admit the following expression:
\begin{gather*}
\widetilde{F}_{1} = \begin{pmatrix}
\overline{F}_{1} & O \\* & *
\end{pmatrix}\!, \qquad
\widetilde{F}_{2} = \begin{pmatrix}
\overline{F}_{2} & O \\ * & *
\end{pmatrix}, \qquad
\widetilde{F}_{\infty} = \begin{pmatrix}
\overline{F}_{\infty} & O \\* & *
\end{pmatrix}\!,
\end{gather*}
where $\overline{F}_{1}$, $\overline{F}_{2}$, $\overline{F}_{\infty}$ are $2\times 2$ matrices.
The matrix $\overline{F}_{\infty}$ is diagonal, which follows from the choice of the parameters $g_1, \dots, g_4$.
We can restrict the $q$-difference equation
\begin{gather}
\check{Y}(qx) =
\bigg( \widetilde{F}_{\infty } + \frac{\widetilde{F}_{1}}{1 - x/(ta_{1})} + \frac{\widetilde{F}_{2}}{1 - x/(ta_{2})} \bigg)\check{Y}(x) \label{eq:chYweTchY}
\end{gather}
of size $6$ to that of size $2$ by choosing the first two components and we write
\begin{gather}
\overline{Y} (qx) = \overline{F}(x) \overline{Y} (x), \qquad
\overline{F}(x) = \overline{F}_{\infty} +
\frac{\overline{F}_{1}}{1 - x/(ta_{1})} + \frac{\overline{F}_{2}}{1 - x/(ta_{2})}. \label{eq:appqmc}
\end{gather}
Then
\begin{gather}
\overline{F}_{\infty} = \begin{pmatrix}
 \chi_1 t a_1 a_2 /\theta_1 & 0 \\ 0 & 1
 \end{pmatrix}\!, \qquad
 \overline{F}_{2} = \overline{F}_{1} | _{a_1 \leftrightarrow a_2}, \nonumber
 \\
\overline{F}_{1} = \frac{a_1}{q y z \theta_1 (a_1-a_2)(\theta_1 - \chi_1 t a_1 a_2 )}
\begin{pmatrix}
 -(y-t a_1 ) f_1^{[1]} a_2^2 & -f_2^{[1]} (y- t a_1 ) a_2 \\[.5ex]
 f_1^{[1]} f_3^{[1]} a_2 & f_2^{[1]} f_3^{[1]}
 \end{pmatrix}\!,\label{eq:olFinfty12}
\end{gather}
where
{\samepage\begin{gather*}
f_1^{[1]} = q^2 \chi_1^2 \chi_2 (y-a_3) (y-a_4) z^2 + (y-t a_2 ) (y \chi_2-\chi_1 t a_1 )
\\ \phantom{f_1^{[1]}=}
{}-q \chi_1 \{ 2 \chi_2 y^2 - (\chi_1 t a_1 +\chi_2 t a_2 +\chi_2 a_3 +\chi_2 a_4 ) y +t ( \theta_1 + \theta_2 ) \} z,
\\
f_2^{[1]} = q \chi_1 \chi_2 a_2 (y-a_3) (y-a_4) z -(y- t a_2 ) (a_2 \chi_2 y-\theta_1),
\\
f_3^{[1]} = q (y a_2 \chi_1-\theta_1) z -a_2 (y-t a_1 ).
\end{gather*}}

It is expected that the $q$-middle convolution induces an integral transformation.
Let $Y(x)$ be a solution to $Y(qx) = B(x)Y(x)$ in equation~(\ref{eq:JSlinDEqxB}) and write
\begin{gather*}
Y(x)=\begin{pmatrix}y_{1}(x) \\y_{2}(x)\end{pmatrix}\!.
\end{gather*}
It follows from Theorem \ref{thm:qcint} that the $q$-difference equation
\begin{gather*}
\widehat{Y}(qx) = F(x)\widehat{Y}(x), \qquad
F(x) = F_{\infty} + \frac{F_{1}}{1 - x/(ta_{1})} + \frac{F_{2}}{1 - x/(ta_{2})}
\end{gather*}
has a solution written as
\begin{gather}
\widehat{Y}(x) =
\begin{pmatrix}
\widehat{Y}_{0}(x) \\
\widehat{Y}_{1}(x) \\
\widehat{Y}_{2}(x)
\end{pmatrix} =
\begin{pmatrix}
\int^{\xi \infty}_{0} s^{-1 } P_{\lambda}(x, s)y_{1}(s) \, {\rm d}_{q}s \\[5pt]
\int^{\xi \infty}_{0} s^{-1 } P_{\lambda}(x, s)y_{2}(s) \, {\rm d}_{q}s \\[5pt]
\int^{\xi \infty}_{0} (s - ta_{1})^{-1 } P_{\lambda}(x, s)y_{1}(s) \, {\rm d}_{q}s \\[5pt]
\int^{\xi \infty}_{0} (s - ta_{1})^{-1 } P_{\lambda}(x, s)y_{2}(s) \, {\rm d}_{q}s \\[5pt]
\int^{\xi \infty}_{0} (s - ta_{2})^{-1 } P_{\lambda}(x, s)y_{1}(s) \, {\rm d}_{q}s \\[5pt]
\int^{\xi \infty}_{0} (s - ta_{2})^{-1 } P_{\lambda}(x, s)y_{2}(s) \, {\rm d}_{q}s
\end{pmatrix}\!. \label{eq:Yhat}
\end{gather}
Set $\check{Y}(x) = P^{-1}\widehat{Y}(x)$.
Then it satisfies equation~(\ref{eq:chYweTchY}).
Write
\begin{gather*}
\check{Y}(x) =
\begin{pmatrix}
\check{y}_{1}(x) \\\check{y}_{2}(x) \\\vdots \\\check{y}_{6}(x)
\end{pmatrix}, \qquad
P^{-1} = \begin{pmatrix}
p_{11} & p_{12} & \cdots & p_{16} \\
p_{21} & p_{22} & \cdots & p_{26} \\
\vdots & \vdots & \ddots & \vdots \\
p_{61} & p_{62} & \cdots & p_{66}
\end{pmatrix}\!.
\end{gather*}
Then the function
$\overline{Y}(x) = \left(\!
\begin{smallmatrix}
\check{y}_1 (x) \\
\check{y}_2 (x)
\end{smallmatrix}\!\right) $
satisfies equation~(\ref{eq:appqmc}).
On the other hand, it follows from equation~(\ref{eq:Yhat}) that
\begin{align*}
\check{y}_{1}(x) &= \int^{\xi \infty}_{0} \bigg\{\bigg( \frac{p_{11}}{s} + \frac{p_{13}}{s - ta_{1}} + \frac{p_{15}}{s - ta_{2}}\bigg)y_{1}(s)
+ \bigg( \frac{p_{12}}{s} + \frac{p_{14}}{s - ta_{1}} + \frac{p_{16}}{s - ta_{2}} \bigg)y_{2}(s) \bigg\}
\\
&\hphantom{=}\times P_{\lambda}(x, s)\, {\rm d}_{q}s.
\end{align*}
By a straightforward calculation, we have
\begin{gather*}
\frac{p_{12}}{s} + \frac{p_{14}}{s - ta_{1}} + \frac{p_{16}}{s - ta_{2}} = \frac{t \theta_{1}}{q w y z \chi_{2} (\chi_{1} t a_{1}a_{2} - \theta_{1})s} b_{12}(s),
\\
\frac{p_{11}}{s} + \frac{p_{13}}{s - ta_{1}} + \frac{p_{15}}{s - ta_{2}} = \frac{t \theta_{1}}{q w y z \chi_{2} (\chi_{1} t a_{1}a_{2} - \theta_{1})s} (b_{11}(s) -1),
\end{gather*}
where $b_{12}(s) $ and $b_{11}(s) $ are elements of the matrix $B(s)$ in equation~(\ref{eq:Bx}).
Therefore
\begin{gather*}
\bigg( \frac{p_{11}}{s} + \frac{p_{13}}{s - ta_{1}} + \frac{p_{15}}{s - ta_{2}} \bigg)y_{1}(s) +
\bigg( \frac{p_{12}}{s} + \frac{p_{14}}{s - ta_{1}} + \frac{p_{16}}{s - ta_{2}} \bigg)y_{2}(s)
\\ \qquad
{} = \frac{t \theta_{1}}{q w y z \chi_{2} (\chi_{1} t a_{1}a_{2} - \theta_{1})s}
\{ - y_{1}(s) + b_{11}(s)y_{1}(s) + b_{12}(s)y_{2}(s) \}.
\end{gather*}
Hence it follows from $ y_{1}(qs) = b_{11}(s)y_{1}(s) + b_{12}(s)y_{2}(s) $ that
\begin{gather*}
\check{y}_{1}(x) = \frac{t \theta_{1}}{q w y z \chi_{2} (\chi_{1} t a_{1}a_{2} - \theta_{1})} \int^{\xi \infty}_{0}\frac{y_{1}(qs) - y_{1}(s)}{s}P_{\lambda}(x, s)\, {\rm d}_{q}s.
\end{gather*}
We are going to obtain the integral representation without using $y_{1}(qs)$.
It follows from the definitions of the Jackson integral and the function $P_{\lambda}(x, s) $ that
\begin{gather*}
\int^{\xi \infty}_{0} y_{1}(qs)P_{\lambda}(x, s)\, \frac{{\rm d}_{q}s}{s} =
 \int^{\xi \infty}_{0} y_{1}(s)P_{\lambda} (x, s/q ) \frac{{\rm d}_{q}s}{s} =
 \int^{\xi \infty}_{0} y_{1}(s)\, \frac{x - q^{\lambda}s}{x - s}P_{\lambda}(x, s)\, \frac{{\rm d}_{q}s}{s}.
\end{gather*}
See \cite{Tkg} for details. Hence
\begin{align*}
\int^{\xi \infty}_{0}\frac{y_{1}(qs) - y_{1}(s)}{s}P_{\lambda}(x, s)\, {\rm d}_{q}s& = \int^{\xi \infty}_{0}\frac{1}{s}\bigg( \frac{x - q^{\lambda}s}{x - s} - 1\bigg)y_{1}(s)P_{\lambda}(x, s)\, {\rm d}_{q}s
\\[1ex]
& = (1 - q^{\lambda})\int^{\xi \infty}_{0}\frac{y_{1}(s)}{x - s}P_{\lambda}(x, s)\, {\rm d}_{q}s.
\end{align*}
We recall that $ q^{\lambda} = \chi_{2} t a_1a_2 /\theta_1 $.
Thus, we obtain
\begin{gather}
 \check{y}_{1}(x) = \frac{t (\chi_{2} t a_1a_2 - \theta_{1})}{q w y z \chi_{2} (\chi_{1} t a_{1}a_{2} - \theta_{1})} \int^{\xi \infty}_{0}\frac{y_{1}(s)}{s - x}P_{\lambda}(x, s)\, {\rm d}_{q}s. \label{eq:checky1}
\end{gather}
We also calculate $\check{y}_{2}(x)$.
It follows from equation~(\ref{eq:Yhat}) that
\begin{align*}
\check{y}_{2 }(x) &= \int^{\xi \infty}_{0} \bigg\{ \bigg( \frac{p_{21}}{s} + \frac{p_{23}}{s - ta_{1}} + \frac{p_{25}}{s - ta_{2}}\bigg)y_{1}(s)
 + \bigg( \frac{p_{22}}{s} + \frac{p_{24}}{s - ta_{1}} + \frac{p_{26}}{s - ta_{2}} \bigg)y_{2}(s) \bigg\}
 \\
 &\hphantom{=}\times P_{\lambda}(x, s)\, {\rm d}_{q}s.
\end{align*}
By a straightforward calculation, we have
\begin{gather*}
\frac{p_{22}}{s} + \frac{p_{24}}{s - ta_{1}} + \frac{p_{26}}{s - ta_{2}} = \frac{\theta_1 \{ ( t a_1 a_2 -q z \theta_1 ) s + y t a_1 a_2 ( q z \chi_1 - 1) \} }{q w y z \chi_2 a_1 a_2 (\chi_1 t a_1 a_2 -\theta_1) (s-y) s} b_{12}(s),
\\
\frac{p_{21}}{s} + \frac{p_{23}}{s - ta_{1}} + \frac{p_{25}}{s - ta_{2}} = \frac{\theta_1 \{ ( t a_1 a_2 -q z \theta_1 ) s + y t a_1 a_2 ( q z \chi_1 - 1) \} }{q w y z \chi_2 a_1 a_2 (\chi_1 t a_1 a_2 -\theta_1) (s-y) s} b_{11}(s)
\\ \hphantom{\frac{p_{21}}{s} + \frac{p_{23}}{s - ta_{1}} + \frac{p_{25}}{s - ta_{2}} = }
{} -\frac{t \{ (t a_1 a_2 - q z \theta_1 ) \chi_1 s + y \theta_1 (q z \chi_1 -1 ) \}}{q w y z \chi_2 (\chi_1 t a_1 a_2 -\theta_1) (s-y) s}.
\end{gather*}
Therefore it follows from $ y_{1}(qs) = b_{11}(s)y_{1}(s) + b_{12}(s)y_{2}(s) $ that
\begin{gather*}
 \bigg( \frac{p_{21}}{s} + \frac{p_{23}}{s - ta_{1}} + \frac{p_{25}}{s - ta_{2}} \bigg)y_{1}(s) +
\bigg( \frac{p_{22}}{s} + \frac{p_{24}}{s - ta_{1}} + \frac{p_{26}}{s - ta_{2}} \bigg)y_{2}(s)
\\ \qquad
{} = \frac{\theta_1 \{ ( t a_1 a_2 -q z \theta_1 ) s + y t a_1 a_2 ( q z \chi_1 - 1) \} }{q w y z \chi_2 a_1 a_2 (\chi_1 t a_1 a_2 -\theta_1) (s-y) s} y_1 (qs)
\\ \qquad\phantom{=}
{} - \frac{t \{ ( t a_1 a_2 - q z \theta_1 ) \chi_1 s + y \theta_1 (q z \chi_1 -1 ) \}}{q w y z \chi_2 (\chi_1 t a_1 a_2 -\theta_1) (s-y) s} y_1(s).
\end{gather*}
We obtain that
\begin{gather*}
\int^{\xi \infty}_{0} \frac{(t a_1 a_2 -q z \theta_1 )s + t a_1 a_2 y (q z \chi_1 -1) }{ s-y } y_1 (qs) P_{\lambda}(x, s)\, \frac{{\rm d}_{q}s}{s}
\\ \qquad
{} = \int^{\xi \infty}_{0}\frac{(t a_1 a_2 -q z \theta_1 ) s + q t a_1 a_2 y ( q z \chi_1 - 1) }{ s - q y } y_{1}(s) P_{\lambda} (x, s/q ) \frac{{\rm d}_{q}s}{s}
\\ \qquad
{}= \int^{\xi \infty}_{0} \frac{(t a_1 a_2 -q z \theta_1 ) s + q t a_1 a_2 y ( q z \chi_1 - 1) }{s - q y } \bigg( 1 + \frac{(1 - q^{\lambda})s}{x - s} \bigg) y_{1}(s) P_{\lambda}(x, s)\, \frac{{\rm d}_{q}s}{s}.
\end{gather*}
Hence
\begin{align}
\check{y}_{2 }(x) &= \frac{1}{q w y z \chi_2 a_1 a_2} \int^{\xi \infty}_{0} \bigg\{ \frac{ q z \theta_1 }{ s-q y } - \frac{ t a_1 a_2 }{ s-y }
+ \bigg( \frac{q z \theta_1 -t a_1 a_2 }{\chi_1 t a_1 a_2 -\theta_1} - \frac{q ^2 y z }{ s - q y } \bigg) \frac{ \chi_2 t a_1 a_2 -\theta_1}{x - s} \bigg\} \nonumber
\\
&\hphantom{=}\times y_{1}(s) P_{\lambda}(x, s)\, {\rm d}_{q}s.\label{eq:checky2}
\end{align}
In summary, we obtain the following theorem by the $q$-middle convolution.

\begin{Theorem} \label{thm:inttrans}
Let $Y(x)$ be a solution to
\begin{gather}
 Y(qx)= \bigg( B_{\infty} + \frac{B_1}{1 - x/(ta_1)} + \frac{B_2}{1 - x /(ta_2 )}\bigg) Y(x), \qquad Y(x)= \begin{pmatrix}y_1 (x) \\y_2 (x)\end{pmatrix}\!, \label{eq:Yqmcthm}
\end{gather}
where $B_{\infty}$, $B_1$ and $B_2$ are defined in equation~\eqref{eq:BinftyB2B1}.
The function
\begin{gather*}
\overline{Y}(x) = \begin{pmatrix}
\check{y}_1 (x) \\ \check{y}_2 (x)
\end{pmatrix}
\end{gather*}
defined by equations~\eqref{eq:checky1} and \eqref{eq:checky2} formally satisfies
\begin{gather}
\overline{Y} (qx) = \bigg(\overline{F}_{\infty} + \frac{\overline{F}_{1}}{1 - x/(ta_{1})} + \frac{\overline{F}_{2}}{1 - x/(ta_{2})}\bigg) \overline{Y} (x), \label{eq:appqmcthm}
\end{gather}
where $\overline{F}_{\infty}$, $\overline{F}_1$ and $\overline{F}_2$ are defined in equation~\eqref{eq:olFinfty12}.
\end{Theorem}

Thus, we obtain the correspondence of the systems of linear $q$-difference equations associated with the $q$-Painlev\'e VI equation by the $q$-middle convolution.
To give the correspondence of the parameters by the $q$-middle convolution in the form of the equation $Y(qx)= \big\{ A_0 (t) + A_1 (t) x\allowbreak +A_2 x^2 \big\} Y(x) $ in equation~(\ref{eq:Ax012}), we need to transform equation~(\ref{eq:appqmcthm}).

Let $\widetilde{c}$ be a non-zero constant which will be fixed later. Set
\begin{gather}
\widetilde{A}(x) = \widetilde{c} (x- t a_1 ) (x- t a_2 ) \bigg( \overline{F}_{\infty} + \frac{\overline{F}_{1}}{1 - x/(ta_{1})} + \frac{\overline{F}_{2}}{1 - x/(ta_{2})} \bigg) \label{eq:Atilde}
\end{gather}
and write $\widetilde{A}(x) = \widetilde{A}(x, t) = \widetilde{A}_{0}(t) + \widetilde{A}_{1}(t)x + \widetilde{A}_{2}x^{2}$. Then we have
\begin{gather*}
\widetilde{A}_{2} = \begin{pmatrix}
\widetilde{c} \chi_{1} t a_{1}a_{2} /\theta_{1} & 0 \\ 0 & \widetilde{c}
\end{pmatrix}\!,
\\[1ex]
\widetilde{A}_{0}(t) \mbox{ has the eigenvalues } \frac{\widetilde{c} \chi_{2} t^{3} a_{1}^{2}a_{2}^{2}}{\theta_{1}} \mbox{ and } \frac{\widetilde{c} t^{2} a_{1}a_{2}\theta_{2}}{\theta_{1}},
\\
\det \widetilde{A}(x, t) =
\frac{\widetilde{c} ^{2}\chi_{1} t a_{1}a_{2}}{\theta_{1}}(x - ta_{1})(x - ta_{2})
\bigg(x - \frac{\chi_{2}t a_{1}a_{2}a_{3}}{\theta_{1}}\bigg) \bigg(x - \frac{\chi_{2} t a_{1}a_{2}a_{4}}{\theta_{1}}\bigg).
\end{gather*}
Hence the action of the $q$-middle convolution to the parameters is described as
\begin{gather}
\chi_{1}\rightarrow \frac{\widetilde{c} \chi_{1} t a_{1}a_{2}}{\theta_{1}}, \qquad
\chi_{2} \rightarrow \widetilde{c}, \qquad
\{ t\theta_{1}, t\theta_{2} \} \rightarrow \bigg\{ \frac{\widetilde{c} t^{2} a_{1}a_{2} \theta_{2}}{\theta_{1}}, \frac{\widetilde{c} \chi_{2} t^{3} a_{1}^{2}a_{2}^{2}}{\theta_{1}} \bigg\},\nonumber
\\
\{a_{1}, a_{2 } \} \rightarrow \{ a_{1}, a_{2 } \}, \qquad
\{ a_{3}, a_{4 } \} \rightarrow \bigg\{ \frac{\chi_{2} t a_{1}a_{2}a_{3}}{\theta_{1}}, \frac{\chi_{2} t a_{1}a_{2}a_{4}}{\theta_{1}} \bigg\}. \label{eq:mctoparameter}
\end{gather}
We investigate the action to the parameters $y$ and $z$.
Recall that $y$ and $z$ are determined by equation~(\ref{eq:wyz}).
We denote the images of $y$ and $z$ by $\widetilde{y}$ and $\widetilde{z}$.
Let $\widetilde{a}_{11} (x)$ (resp.~$\widetilde{a}_{12} (x)$) be the upper left entry (resp.~the upper right entry) of the matrix $\widetilde{A}(x)$.
Then the value $\widetilde{y}$ is the zero of the linear function $\widetilde{a}_{12} (x)$, and we have
\begin{align}
\widetilde{y} &=
\frac{\chi_{2} t a_{1}a_{2} \{ q\chi_{1}(y - a_{3})(y - a_{4})z - (y - ta_{1})(y - ta_{2}) \}}{q\chi_{1}\chi_{2} t a_{1}a_{2} (y - a_{3})(y - a_{4})z - \theta_{1}(y - ta_{1})(y - ta_{2})}\, y \nonumber
\\
&= \frac{qz\dfrac{(y-a_{3})(y-a_{4})}{(y-ta_{1})(y-ta_{2})} - \dfrac{1}{\chi_{1}}}{ qz\dfrac{(y-a_{3})(y-a_{4})}{(y-ta_{1})(y-ta_{2})} - \dfrac{\theta_{1}}{\chi_{1}\chi_{2}t a_{1}a_{2}}}\, y.\label{ytilde}
\end{align}
The value $\widetilde{z}$ satisfies $\widetilde{a}_{11} (x) | _{x=\widetilde{y} } = (\widetilde{y} - t a_{1})(\widetilde{y} -t a_{2}) /(q \widetilde{z}) $.
Since
\begin{gather*}
\widetilde{a}_{11} (x) | _{x=\widetilde{y} }= \frac{\widetilde{c}}{q \chi _2 z} \frac{(y - ta_{1})(y - ta_{2})}{(y - a_{3})(y - a_{4})} \bigg( \widetilde{y} - \frac{\chi_{2}t a_{1}a_{2}a_{3}}{\theta_{1}}\bigg) \bigg( \widetilde{y} - \frac{\chi_{2}t a_{1}a_{2}a_{4}}{\theta_{1}} \bigg),
\end{gather*}
we have
\begin{gather}
\widetilde{z} = z \frac{\chi _2 }{\widetilde{c} }
\frac{(y-a_{3})(y-a_{4})}{(y-ta_{1})(y-ta_{2})}
\frac{(\widetilde{y} - ta_{1})(\widetilde{y} - ta_{2})}
{( \widetilde{y} - \chi_{2} t a_{1}a_{2}a_{3}/\theta_{1})
( \widetilde{y} - \chi_{2} t a_{1}a_{2}a_{3}/\theta_{1})}. \label{eq:ztilde}
\end{gather}

\section[q-middle convolution and Weyl group symmetry of q-Painleve VI equation]
{$\boldsymbol q$-middle convolution and Weyl group symmetry \\ of $\boldsymbol q$-Painlev\'e VI equation} \label{sec:qmcWeyl}

We investigate the transformation of the parameters induced by the $q$-middle convolution in terms of the Weyl group symmetry associated with the $q$-Painlev\'e VI equation.
Kajiwara, Noumi and Yamada gave a survey on discrete Painlev\'e equations in \cite{KNY}.
They presented a list of the Weyl group symmetry and a Lax pair for each discrete Painlev\'e equation.
The $q$-Pain\-lev\'e~VI equation by Jimbo and Sakai \cite{JS} corresponds to the equation $q$-$P\big(D^{(1)}_5\big)$ in~\cite{KNY}.
In this section, we make a correspondence between the parameters of the $q$-Painlev\'e VI equation and those of the equation $q$-$P\big(D^{(1)}_5\big)$.

The equation $q$-$P\big(D^{(1)}_5\big)$ was obtained by the compatibility condition of the Lax pair $L_1$ and~$L_2$ in~\cite{KNY}.
The operator $L_1$ is defined by
\begin{align}
L_{1} y(x) &= \bigg\{\frac{x(g\nu_{1}-1)(g\nu_{2}-1)}{qg}-\frac{\nu_{1}\nu_{2}\nu_{3}\nu_{4}(g- \nu_{5}/\kappa _{2})(g- \nu_{6}/\kappa _{2})}{fg}\bigg\} y(x) \nonumber
\\
&\hphantom{=} + \frac{\nu_{1}\nu_{2}(x- q \nu_{3})(x - q \nu_{4})}{q(qf-x)}( g y(x) - y ( x/q ) ) \nonumber
\\
&\hphantom{=} + \frac{(x- \kappa _{1} /\nu_{7})(x - \kappa _{1}/\nu_{8})}{q(f-x)}\bigg( \dfrac{1}{g}y(x) - y(qx) \bigg),\label{eq:KNYlaxpair}
\end{align}
which is independent from the time evolution.
Here the parameters are constrained by the relation{\samepage
\begin{gather}
\kappa _{1}^{2} \kappa _{2}^{2} = q \nu_1 \nu_2 \nu_3 \nu_4 \nu_5 \nu_6 \nu_7 \nu_8. \label{eq:relation2}
\end{gather}
In this paper, we do not use the operator $L_2$, which contains the operation of the time evolution.}

The correspondence between the parameters of the $q$-Painlev\'e VI equation and those of the equation $q$-$P\big(D^{(1)}_5\big)$ was made by considering the linear $q$-difference equation $Y(qx)=A(x) Y(x)$ in equation~(\ref{eq:JSlinDEqx1}) and $L_{1}y(x) =0$ in equation~(\ref{eq:KNYlaxpair}).

For the system of $q$-difference equation
\begin{gather*}
Y(qx) = A(x )Y(x), \qquad
A(x) = \begin{pmatrix}
a_{11}(x) & a_{12}(x) \\ a_{21}(x) & a_{22}(x)
\end{pmatrix}\!, \qquad
Y(x) = \begin{pmatrix}
y_{1}(x) \\ y_{2}(x)
\end{pmatrix}\!,
\end{gather*}
we calculate the $q$-difference equation for $y_{1}(x)$ by eliminating $y_2 $.
The system of the $q$-difference equation is written as
\begin{gather*}
y_{1}(qx) = a_{11}(x)y_{1}(x) + a_{12}(x)y_{2}(x),
\\
y_{2}(qx) = a_{21}(x)y_{1}(x) + a_{22}(x)y_{2}(x).
\end{gather*}
We substitute $y_{2}(x)\! = a_{21}(x/q)y_{1}(x/q) + a_{22}(x/q)y_{2}(x/q)$ into $y_{1}(qx)\! = a_{11}(x)y_{1}(x) + a_{12}(x)y_{2}(x)$.
Then we have
\begin{gather*}
y_{1}(qx) = a_{11}(x)y_{1}(x) + a_{12}(x)a_{21} (x/q) y_{1} (x/q) + a_{12}(x)a_{22}(x/q)y_{2}(x/q).
\end{gather*}
We elimilate $y_{2}(x/q)$ by the relation $y_2 (x/q) = \{y_1(x)-a_{11}(x/q) y_1(x/q)\} /a_{12}(x/q) $.
Thus,
\begin{gather*}
\frac{y_{1}(qx)}{a_{12}(x)} - \bigg\{\frac{a_{11}(x)}{a_{12}(x)} + \frac{a_{22}(x/q)}{a_{12}(x/q)}\bigg\}y_{1}(x)
+ \frac{a_{11}(x/q) a_{22}(x/q)- a_{12}(x/q) a_{21}(x/q)}{a_{12}(x/q)}y_{1}(x/q) = 0.
\end{gather*}
We restrict it to the case that the matrix elements are fixed to equation~(\ref{eq:JSlaxpair}).
Then it follows from equations~(\ref{eq:Ax012}) and (\ref{eq:JSlaxpair}) that
\begin{gather}
\frac{y_{1}(qx)}{x - y} -
\bigg\{\frac{a_{11}(x)}{x - y} + \frac{a_{22} (x/q)}{x/q - y}\bigg\}y_{1}(x) \nonumber
\\ \qquad
{}+ \frac{\chi_{1}\chi_{2}(x/q - ta_{1})(x/q - ta_{2})(x/q - a_{3})(x/q - a_{4})}{x/q - y}y_{1} (x/q) = 0. \label{eq:y1qxxxq}
\end{gather}
Let $\phi (x)$ be the function such that $\phi (qx) = d_0 (x- t a_1 ) (x -ta_2) \phi (x) $.
Set $u (x)= y_1 (x) / \phi (x) $.
Then it follows from equation~(\ref{eq:y1qxxxq}) that the function $u(x)$ satisfies
\begin{gather*}
\frac{(x - ta_{1})(x - ta_{2})}{x - y} d_0 u(qx) - \bigg\{\frac{a_{11}(x)}{x - y} + \frac{a_{22}(x/q )}{x/q - y}\bigg\} u(x)
\\ \qquad
{} + \frac{\chi_{1}\chi_{2}(x/q - a_{3})(x/q - a_{4})}{x/q - y} \frac{1}{d_0} u (x/q) = 0.
\end{gather*}
We rewrite it by using equation~(\ref{eq:JSlaxpair}) as
\begin{gather}
\frac{(x - ta_{1})(x - ta_{2})}{x - y}\bigg\{ d_0 u (qx) - \frac{u (x)}{qz}\bigg\}
 + \frac{\chi_{1}\chi_{2}(x/q - a_{3})(x/q - a_{4})}{x/q - y} \bigg\{ \frac{1}{d_0} u (x/q) - qz u (x)\bigg\} \nonumber
 \\[1ex] \qquad
 {}- \bigg\{\chi_{1}(x - \alpha) + \frac{\chi_{1}z_{1}}{x - y} + \chi_{2} (x/q - \beta ) + \frac{\chi_{2}z_{2}}{x/q - y} - \frac{(x - ta_{1})(x - ta_{2})}{qz(x - y)} \nonumber
 \\[3pt] \qquad \hphantom{- \bigg\{}
 {}- \frac{qz\chi_{1}\chi_{2} (x/q - a_{3})(x/q - a_{4})}{x/q - y} \bigg\} u (x) = 0. \label{eq:3}
\end{gather}
It is seen that the poles $x=y$ and $x=qy$ are cancelled on the coefficient of $u (x)$.
By a~straight\-forward calculation (see \cite{Tkg} for details), equation~(\ref{eq:3}) is written as
{\samepage\begin{gather}
\bigg\{ \frac{x(q z \chi_{1} - 1)(z \chi_{2} - 1)}{q^{2}z} - \chi_{1}\chi_{2}a_{3}a_{4}\frac{(qz - ta_{1}a_{2}/\theta_{1}) (qz - ta_{1}a_{2}/\theta_{2})}{q^2 y z}\bigg\} u (x) \nonumber
\\ \qquad
{}+ \frac{\chi_{1}\chi_{2}(a_3 - x/q ) (a_4 - x/q)}{q y - x} \bigg\{qz u (x) - \frac{1}{d_0} u (x/q) \bigg\} \nonumber
\\ \qquad
{}+ \frac{(x - ta_{1})(x - ta_{2})}{q(y -x)} \bigg\{ \frac{u (x)}{qz} - d_0 u (qx) \bigg\} = 0. \label{eq:y1second}
\end{gather}}
We compare equation~(\ref{eq:y1second}) for the case $d_0 =1$ with equation~(\ref{eq:KNYlaxpair}).
Then we obtain the following correspondence:
\begin{gather}
qz = g, \qquad y = f, \qquad \chi_{1} = \nu_{1}, \qquad \chi_{2} = q\nu_{2}, \qquad \frac{ta_{1}a_{2}}{\theta_{1}} = \frac{\nu_{5}}{\kappa _{2}}, \qquad \frac{ta_{1}a_{2}}{\theta_{2}} = \frac{\nu_{6}}{\kappa _{2}}, \nonumber
\\
a_{3} = \nu_{3}, \qquad a_{4} = \nu_{4}, \qquad ta_{1} = \frac{\kappa _{1}}{\nu_{7}}, \qquad ta_{2} = \frac{\kappa _{1}}{\nu_{8}}.\label{eq:JSKNYcorrespondence}
\end{gather}
On the relation of the parameters, the condition $\chi_{1}\chi_{2}a_{1}a_{2}a_{3}a_{4} = \theta_{1} \theta_{2} $ is equivalent to equation~(\ref{eq:relation2}).
Thus, we obtained a correspondence between the parameters of the $q$-Painlev\'e VI equation in \cite{JS} and those of the equation $q$-$P\big(D^{(1)}_5\big)$ in \cite{KNY}.

Sakai \cite{Sak} established that each discrete Painlev\'e equation has the symmetry in terms of the affine Weyl group, and the description of the symmetry was reviewed explicitly by Kajiwara, Noumi and Yamada in \cite{KNY}.
The $q$-Painlev\'e VI equation has the symmetry of the affine Weyl group of the type $D^{(1)}_5$.
We describe the action of the operators $s_0, \dots,s_5$ for the parameters $(\kappa _{1}, \kappa _{2}, \nu_{1}, \dots, \nu_{8}) \in (\mathbb{C}^{\times})^{10}$ and $(f, g) \in \mathbb{P}^{1} \times \mathbb{P}^{1}$ as follows
\begin{gather}
s_0 \colon\ \nu_7 \leftrightarrow \nu_8, \qquad
s_1 \colon\ \nu_3 \leftrightarrow \nu_4, \qquad
s_4 \colon\ \nu_1 \leftrightarrow \nu_2, \qquad
s_5 \colon\ \nu_5 \leftrightarrow \nu_6, \nonumber
\\
 s_2 \colon\ \nu_3 \rightarrow\frac{k_1}{\nu_7},\qquad
 \nu_7 \rightarrow \frac{k_1}{\nu_3},\qquad
 k_2 \rightarrow \frac{k_1k_2}{\nu_3\nu_7},\qquad
 g \rightarrow g\, \frac{f - \nu_3}{f - k_1 /\nu_7}, \nonumber
 \\
 s_3 \colon\ \nu_1 \rightarrow\frac{k_2}{\nu_5},\qquad
 \nu_5 \rightarrow \frac{k_2}{\nu_1}, \qquad
 k_1 \rightarrow \frac{k_1k_2}{\nu_1\nu_5},\qquad
 f \rightarrow f\, \frac{g - 1/\nu_1}{g - \nu_5/k_2}.\label{eq:AffineWeylAction}
\end{gather}
The omitted variables are invariant by the action, i.e.,~$s_2 (f) =f$.
Then we can confirm that these operations satisfy the relations of the Weyl group $W\big(D^{(1)}_{5}\big)$ whose Dynkin diagram is as follows

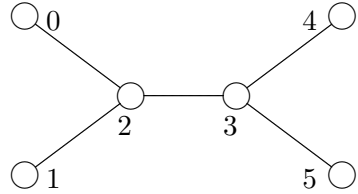
\hspace*{30mm}\begin{picture}(0,85)(0,10)
\put(70,80){\circle{10}}
\put(70,20){\circle{10}}
\put(190,80){\circle{10}}
\put(190,20){\circle{10}}
\put(110,50){\circle{10}}
\put(150,50){\circle{10}}
\qbezier(74,77)(90,65)(106,53)
\qbezier(74,23)(90,35)(106,47)
\qbezier(186,77)(170,65)(154,53)
\qbezier(186,23)(170,35)(154,47)
\qbezier(115,50)(130,50)(145,50)
\put(78,75){$0$}
\put(78,15){$1$}
\put(175,75){$4$}
\put(175,15){$5$}
\put(105,35){$2$}
\put(145,35){$3$}
\end{picture}

On the other hand, the $q$-middle convolution induces the transformation of the parameters of the $q$-Painlev\'e VI equation given in equation~(\ref{eq:mctoparameter}), although there was an arbitrary parameter $\widetilde{c}$.
We describe it in terms of the Weyl group action by specializing the parameter $\widetilde{c}$ in equation~(\ref{eq:Atilde}).

\begin{Proposition}
We specify the parameter $\widetilde{c}$ in equation~\eqref{eq:Atilde} by setting $\widetilde{c} =\chi _2$.
Then the transformation of the parameters of the $q$-Painlev\'e VI equation which is induced by the $q$-middle convolution coincides with the action
\begin{gather}
s_{5}s_{2}s_{1}s_{0}s_{2}s_{3}s_{2}s_{0}s_{1}s_{2}
\label{eq:s5210232012}
\end{gather}
by the generators of $W\big(D^{(1)}_{5}\big) $, and it is written as
\begin{gather}
 \nu_{1} \rightarrow q\, \frac{\nu_{1}\nu_{2}\nu_{5}}{\kappa _{2}}, \qquad
 \nu_{2} \rightarrow \nu_{2}, \qquad
 \frac{\kappa _{1}}{\nu_{7}} \rightarrow \frac{\kappa _{1}}{\nu_{7}}, \qquad
 \frac{\kappa _{1}}{\nu_{8}} \rightarrow \frac{\kappa _{1}}{\nu_{8}}, \nonumber
 \\
 \nu_{3} \rightarrow q\, \frac{\nu_{2}\nu_{3}\nu_{5}}{\kappa _{2}}, \qquad
 \nu_{4} \rightarrow q\, \frac{\nu_{2}\nu_{4}\nu_{5}}{\kappa _{2}}, \qquad
 \frac{\kappa _{2}}{\nu_{5}} \rightarrow q\, \frac{\nu_{2}\nu_{5}}{\nu_{6}}, \qquad
 \frac{\kappa _{2}}{\nu_{6}} \rightarrow q^{2}\, \frac{\nu_{2}^{2}\nu_{5}}{\kappa _{2}}, \nonumber
 \\
 f \rightarrow \widetilde{f} = \frac{ \dfrac{(f - \nu_{3})(f -\nu_{4})}{(f - \kappa _{1}/\nu_{7}) (f - \kappa _{1}/\nu_{8})}g - \dfrac{1}{\nu_{1}}}{ \dfrac{(f - \nu_{3})(f -\nu_{4})}{(f - \kappa _{1}/\nu_{7} ) (f - \kappa _{1}/\nu_{8} )}g - \dfrac{\kappa _{2}}{q\nu_{1}\nu_{2}\nu_{5}}}\, f, \nonumber
 \\
 g \rightarrow \widetilde{g} = \frac{(f - \nu_{3})(f - \nu_{4})}{(f - \kappa _{1}/\nu_{7} ) (f - \kappa _{1}/\nu_{8} )}\frac{(\widetilde{f} - \kappa _{1}/\nu_{7} ) (\widetilde{f} - \kappa _{1}/\nu_{8} )}{(\widetilde{f} - q \nu_{2}\nu_{3}\nu_{5}/\kappa _{2} ) (\widetilde{f} - q \nu_{2}\nu_{4}\nu_{5} /\kappa _{2} )}\, g.\label{eq:qmcKNY}
\end{gather}
\end{Proposition}

\begin{proof}
Set $\widetilde{c} =\chi _2 $.
Then it follows from equations~(\ref{eq:mctoparameter}), (\ref{ytilde}) and (\ref{eq:ztilde}) that we may write the transformation of the parameters induced by the $q$-middle convolution as
\begin{gather}
 \chi_{1}\rightarrow \frac{\chi_{1} \chi_{2} t a_{1}a_{2}}{\theta_{1}}, \qquad
 \chi_{2} \rightarrow \chi _2, \qquad
 a_{1} \rightarrow a_{1},\qquad
 a_{2} \rightarrow a_{2 }, \nonumber
 \\
 a_{3} \rightarrow \frac{\chi_{2}t a_{1}a_{2}a_{3}}{\theta_{1}}, \qquad
 a_{4 } \rightarrow \frac{\chi_{2} t a_{1}a_{2}a_{4}}{\theta_{1}}, \qquad
 t\theta_{1} \rightarrow \frac{\chi_{2} t^{2} a_{1}a_{2}\theta_{2}}{\theta_{1}}, \qquad
 t\theta_{2} \rightarrow \frac{\chi_{2}^2 t^{3} a_{1}^{2}a_{2}^{2}}{\theta_{1}},\nonumber
 \\
 y \rightarrow \widetilde{y} = \frac{ qz\dfrac{(y-a_{3})(y-a_{4})}{(y-ta_{1})(y-ta_{2})} - \dfrac{1}{\chi_{1}}}{ qz\dfrac{(y-a_{3})(y-a_{4})}{(y-ta_{1})(y-ta_{2})} - \dfrac{\theta_{1}}{\chi_{1}\chi_{2} t a_{1}a_{2}}}\, y, \nonumber
 \\
 z \rightarrow \widetilde{z} = \frac{(y-a_{3})(y-a_{4})}{(y-ta_{1})(y-ta_{2})}\frac{(\widetilde{y} - ta_{1})(\widetilde{y} - ta_{2})}{( \widetilde{y} - \chi_{2} t a_{1}a_{2}a_{3}/\theta_{1}) (\widetilde{y} - \chi_{2} t a_{1}a_{2}a_{3}/\theta_{1} )} z.\label{eq:mctoparameterk2}
\end{gather}
By the correspondence in equation~(\ref{eq:JSKNYcorrespondence}), it is rewritten as equation~(\ref{eq:qmcKNY}).

We show that the transformation of the parameters given in equation~(\ref{eq:qmcKNY}) coincides with the consequence of the action given in equation~(\ref{eq:s5210232012}).
We apply the operation $s_{2}s_{0}s_{1}s_{2}$ to $f$ and~$g$.
Then
\begin{gather*}
 s_{2}s_{0}s_{1}s_{2}(f) = f, \qquad
 s_{2}s_{0}s_{1}s_{2}(g) = \frac{(f - \nu_{3} )( f - \nu_{4} )}{(f - \kappa _{1}/\nu_{7}) (f - \kappa _{1}/\nu_{8})}\, g.
\end{gather*}
Note that we define the composition of the transformations as automorphisms of the algebra (symbolical composition in \cite[Remark 2.1]{KNY}).
Since $s_3 (f) = f (g - 1/\nu_1 )/( g - \nu_5/k_2)$, we have
\begin{gather*}
s_{2}s_{0}s_{1}s_{2}s_{3}(f) =
\frac{ \dfrac{(f - \nu_{3})(f -\nu_{4})}{(f - \kappa _{1}/\nu_{7}) (f - \kappa _{1}/\nu_{8})}g - \dfrac{1}{\nu_{1}}}{\dfrac{(f - \nu_{3})(f -\nu_{4})}{(f - \kappa _{1}/\nu_{7} ) (f - \kappa _{1}/\nu_{8} )}g - \dfrac{\kappa _{2}}{q\nu_{1}\nu_{2}\nu_{6}}}\, f.
\end{gather*}
Here we used equation~(\ref{eq:relation2}).
By comparing it with equation~(\ref{eq:qmcKNY}), we obtain
\begin{gather*}
s_{5}s_{2}s_{0}s_{1}s_{2}s_{3}(f) = \widetilde{f}.
\end{gather*}
We consider the operation to $g$.
Since $g$ is invariant under the actions of $s_3$ and $s_5$, we have
\begin{gather*}
s_{5}s_{2}s_{0}s_{1}s_{2}s_{3}(g) = s_{2}s_{0}s_{1}s_{2}(g) = \frac{(f - \nu_{3} )( f - \nu_{4} )}{(f - \kappa _{1}/\nu_{7}) (f - \kappa _{1}/\nu_{8})}\, g.
\end{gather*}
We set $s = s_{5}s_{2}s_{0}s_{1}s_{2}s_{3}$.
Then
\begin{align*}
s_{5}s_{2}s_{1}s_{0}s_{2}s_{3}s_{2}s_{0}s_{1}s_{2}(g) &=
\frac{s(f) - s(\nu_{3})}{s(f) - s(\kappa _{1}/\nu_{7})}
\frac{s(f) - s(\nu_{4})}{s(f) - s(\kappa _{1}/\nu_{8})}\, s(g)
\\
& = \frac{\widetilde{f} - s(\nu_{3})}{\widetilde{f} - s \kappa _{1}/\nu_{7})}
\frac{\widetilde{f} - s(\nu_{4})}{\widetilde{f} - s(\kappa _{1}/ \nu_{8})}
\frac{f - \nu_{3}}{f - \kappa _{1}/\nu_{7}}\frac{f - \nu_{4}}{f - \kappa _{1}/\nu_{8} }\, g.
\end{align*}
It follows from equation~(\ref{eq:relation2}) that
\begin{gather*}
s\bigg(\frac{\kappa _{1}}{\nu_{7}}\bigg) = \frac{q\nu_{2}\nu_{4}\nu_{5}}{\kappa _{2}}, \qquad
s\bigg(\frac{\kappa _{1}}{\nu_{8}}\bigg) = \frac{q\nu_{2}\nu_{3}\nu_{5}}{\kappa _{2}}, \qquad
s(\nu_{3}) = \frac{\kappa _{1}}{\nu_{8}}, \qquad
s(\nu_{4}) = \frac{\kappa _{1}}{\nu_{7}}.
\end{gather*}
Hence
\begin{gather*}
s_{5}s_{2}s_{1}s_{0}s_{2}s_{3}s_{2}s_{0}s_{1}s_{2}(g) = \widetilde{g},
\\
s_{5}s_{2}s_{1}s_{0}s_{2}s_{3}s_{2}s_{0}s_{1}s_{2}(f) = s_{5}s_{2}s_{0}s_{1}s_{2}s_{3}(f) = \widetilde{f}.
\end{gather*}
The action of the operation $s_{5}s_{2}s_{1}s_{0}s_{2}s_{3}s_{2}s_{0}s_{1}s_{2}$ to the other parameters is given by
\begin{gather*}
\kappa _{1} \rightarrow \frac{q\kappa _{1}\nu_{2}\nu_{5}}{\kappa _{2}}, \qquad
 \kappa _{2} \rightarrow \frac{q^{2}\nu_{2}^{2}\nu_{5}^{2}}{\kappa _{2}}, \qquad
 \\
\nu_{1} \rightarrow \frac{q\nu_{1}\nu_{2}\nu_{5}}{\kappa _{2}}, \qquad
 \nu_{2} \rightarrow \nu_{2}, \qquad
 \nu_{3} \rightarrow \frac{q\nu_{2}\nu_{3}\nu_{5}}{\kappa _{2}}, \qquad
 \nu_{4} \rightarrow \frac{q\nu_{2}\nu_{4}\nu_{5}}{\kappa _{2}}, \qquad
 \\
\nu_{5} \rightarrow \frac{q\nu_{2}\nu_{5}\nu_{6}}{\kappa _{2}}, \qquad
 \nu_{6} \rightarrow \nu_{5}, \qquad
 \nu_{7} \rightarrow \frac{q\nu_{2}\nu_{5}\nu_{7}}{\kappa _{2}}, \qquad
 \nu_{8} \rightarrow \frac{q\nu_{2}\nu_{5}\nu_{8}}{\kappa _{2}},
\end{gather*}
and it recovers equation~(\ref{eq:qmcKNY}).
\end{proof}

\section[Integral transformation on q-Heun equation]
{Integral transformation on $\boldsymbol q$-Heun equation} \label{sec:ITHeun}

We interpret the integral transformation in Theorem \ref{thm:inttrans} as the one for solutions of the single second-order linear $q$-difference equations.

\begin{Proposition} \label{prop:y1cy1}\qquad
\begin{enumerate}\itemsep=0pt
\item[$(i)$] The function $y_1 (x)$ in equation~\eqref{eq:Yqmcthm} satisfies
\begin{gather}
\bigg\{\frac{x(q \chi_{1} z - 1)(\chi_{2}z - 1)}{z} - \chi_{1}\chi_{2}a_{3}a_{4}\frac{(qz - ta_{1}a_{2}/\theta_{1} )(qz - ta_{1}a_{2}/\theta_{2})}{y z}\bigg\} y_1 (x) \nonumber
\\ \qquad
{}+ \frac{\chi_{1}\chi_{2}(x - q a_3)( x - q a_4)}{qy - x} \bigg\{ qz y_1 (x) - \frac{t a_1 a_2 }{\theta _1} y_1 ( x/q ) \bigg\} \nonumber
\\[1ex] \qquad
{}+ \frac{q(x - ta_{1})(x - ta_{2})}{y -x} \bigg\{ \frac{y_1 (x)}{qz} - \frac{\theta _1}{t a_1 a_2 } y_1 (qx) \bigg\} = 0.\label{eq:y1secondIT}
\end{gather}
\item[$(ii)$] The function $\check{y}_1 (x)$ in equation~\eqref{eq:appqmcthm} satisfies
\begin{gather}
\bigg\{ \frac{x(q t \theta_{2} \widetilde{z} - a_3 a_4)(\chi_{2}\widetilde{z} - 1)}{a_3 a_4 \widetilde{z}} - \frac{t^2 a_1 a_2 (q \chi_{2}\theta _2 \widetilde{z} - \theta _1) \big(q\chi_{2}^2 t a_{1}a_{2} \widetilde{z} - \theta_{1} \big)}{\theta _1^2 \widetilde{y}\widetilde{z}}\bigg\} \check{y}_1 (x) \nonumber
\\ \qquad
{}+ \frac{\chi_{2} t \theta_{2} (x- q t \theta_{2}/(\chi_{1}a_{4}))(x- q t \theta_{2}/(\chi_{1}a_{3}))}{a_3 a_4 (q\widetilde{y} - x)} \bigg\{ q\widetilde{z} \check{y}_1 (x) - \frac{1}{\chi_{2}} \check{y}_1 ( x/q ) \bigg\} \nonumber
\\ \qquad
{}+ \frac{q(x - ta_{1})(x - ta_{2})}{\widetilde{y} -x} \bigg\{ \frac{\check{y}_1 (x)}{q\widetilde{z}} - \chi_{2} \check{y}_1 (qx) \bigg\} = 0,\label{eq:cy1secondIT}
\end{gather}
where $\widetilde{y} $ and $\widetilde{z} $ are determined by equation~\eqref{eq:mctoparameterk2}.
\end{enumerate}
\end{Proposition}
\begin{proof}
It follows from equation~(\ref{eq:JSlinDEqxB}) that the function $y_1 (x)$ in equation~(\ref{eq:Yqmcthm}) satisfies equation~(\ref{eq:y1second}) with the condition $d_0=c_0 = \theta _1 /(t a_1 a_2 )$.
Then we obtain $(i)$.

Equation~(\ref{eq:appqmcthm}) is obtained as the form of equation~(\ref{eq:Yqmcthm}) by replacing the parameters as $(y,z) \to (\widetilde{y},\widetilde{z} )$ and equation~(\ref{eq:mctoparameter}) up to the ambiguity of the parameter $w$, and the $q$-difference equation for $\check{y}_1 (x) $ does not depend on the parameter $w$.
Hence $(ii)$ follows from equation~(\ref{eq:Atilde}) with the condition $\widetilde{c} = \chi _2$.
\end{proof}

\begin{Theorem} \label{thm:qmcint1}
Assume that $y_1 (x) $ is a solution to equation~\eqref{eq:y1secondIT} and $\lambda $ satisfies $ q^{\lambda} = \chi_{2}a_1a_2t /\theta_1 $.
Then the function
\begin{gather}
\check{y}_{1}(x) =\int^{\xi \infty}_{0}\frac{y_{1}(s)}{s - x}P_{\lambda}(x, s)\, {\rm d}_{q}s
\label{eq:checky1thm}
\end{gather}
formally satisfies equation~\eqref{eq:cy1secondIT}.
\end{Theorem}

\begin{proof}
Let $b_{ij}(x) $ $(i,j \in \{1,2 \} )$ be the elements of the matrix $B(x)$ in equation~(\ref{eq:Bx}).
The function $y_1 (x) $ is given as a solution to equation~(\ref{eq:y1secondIT}).
Define the function $y_2 (x)$ by $ y_{1}(qx) = b_{11}(x)y_{1}(x) + b_{12}(x)y_{2}(x) $.
Since equation~(\ref{eq:y1secondIT}) is written as
\begin{gather*}
\frac{y_{1}(qx)}{b_{12}(x)} - \bigg\{\frac{b_{11}(x)}{b_{12}(x)} + \frac{b_{22}(x/q)}{b_{12}(x/q)}\bigg\}y_{1}(x) + \frac{b_{11}(x/q) b_{22}(x/q)- b_{12}(x/q) b_{21}(x/q)}{b_{12}(x/q)}y_{1}(x/q) = 0,
\end{gather*}
we obtain the equality $ y_{2}(x) = b_{21}(x/q)y_{1}(x/q) + b_{22}(x/q)y_{2}(x/q) $.
Then the function $Y(x)= {}^{t} (y_1 (x), y_2 (x))$ satisfies equation~(\ref{eq:Yqmcthm}).
Hence it follows from Theorem \ref{thm:qcint} that the function $\check{Y}(x)= {}^{t} (\check{y}_1 (x), \check{y}_2 (x))$ satisfies equation~(\ref{eq:appqmcthm}).
Therefore the function $\check{y}_{1}(x)$ in equation~(\ref{eq:checky1thm}) satisfies equation~(\ref{eq:cy1secondIT}) by Proposition~\ref{prop:y1cy1}$(ii)$.
\end{proof}

\begin{Corollary}
Let $\mu $, $\mu ' $ and $\lambda $ be the constants such that $q^{\mu }= \nu_{5}/\kappa _{2} $, $q^{\mu '}= q \nu_{2}$ and $ q^{\lambda} = q \nu_{2} \nu_{5}/\kappa _{2} $.
Let $y(x)$ be a solution to the equation $L_1 y(x)=0$ which was described in equation~\eqref{eq:KNYlaxpair}.
Then the function
\begin{gather}
\check{y} (x) =
x^{\mu '} \int^{\xi \infty}_{0}\frac{y (s)}{s - x} s ^{\mu } P_{\lambda}(x, s)\, {\rm d}_{q}s
\label{eq:corchecky}
\end{gather}
formally satisfies the equation $\widetilde{L_1} \check{y}(x)=0$, where the operator $\widetilde{L_1}$ is obtained from $L_1$ by replacing the parameters in accordance with the action of $s_{5}s_{2}s_{1}s_{0}s_{2}s_{3}s_{2}s_{0}s_{1}s_{2} $ $($see equation~$(\ref{eq:qmcKNY}))$.
\end{Corollary}
\begin{proof}
Let $y(x)$ be a solution to the equation $L_1 y(x)=0$ (see equation~(\ref{eq:KNYlaxpair})).
Set $y_1(x)= x^{\mu } y(x)$.
Then function $y_1 (x)$ satisfies
\begin{gather*}
\bigg\{\frac{x(g\nu_{1}-1)(g\nu_{2}-1)}{qg}-\frac{\nu_{1}\nu_{2}\nu_{3}\nu_{4}(g- \nu_{5}/\kappa _{2})(g- \nu_{6}/\kappa _{2})}{fg}\bigg\} y_1(x)
 \\ \qquad
{}+ \frac{\nu_{1}\nu_{2}(x- q \nu_{3})(x - q \nu_{4})}{q(qf-x)}\bigg( g y_1(x) - \frac{\nu_{5}}{\kappa _{2}} y_1 ( x/q ) \bigg)
\\ \qquad
{} + \frac{(x- \kappa _{1} /\nu_{7})(x - \kappa _{1}/\nu_{8})}{q(f-x)}\bigg( \dfrac{1}{g}y_1(x) - \frac{\kappa _{2}}{\nu_{5}} y_1(qx) \bigg) =0,
\end{gather*}
and it is written as equation~(\ref{eq:y1secondIT}) by the correspondence in equation~(\ref{eq:JSKNYcorrespondence}).
It follows from Theorem \ref{thm:qmcint1} that the function
\begin{gather*}
\check{y}_{1}(x) = \int^{\xi \infty}_{0}\frac{y (s)}{s - x} s ^{\mu } P_{\lambda}(x, s)\, {\rm d}_{q}s
\end{gather*}
satisfies equation~(\ref{eq:cy1secondIT}).
We apply the correspondence in equation~(\ref{eq:JSKNYcorrespondence}) for the parameters and set $\check{y} (x) =
x^{\mu '} \check{y}_{1}(x)$.
Then it is seen that the function $\check{y} (x) $ is written as equation~(\ref{eq:corchecky}) and satifies the equation $\widetilde{L_1} \check{y}(x)=0$.
\end{proof}
We specialize the parameters $y$ and $z$ in Theorem \ref{thm:qmcint1} as
\begin{gather}
y=a_3, \qquad
z= \frac{( t a_1 -a_3) (t a_2 -a_3)}{q t (\theta_1+\theta_2) +a_3^2 (q \chi_1 + \chi_2)+ E q \theta_1 a_3 /(t a_1 a_2 )}. \label{eq:yzqHeun}
\end{gather}
Then equation~(\ref{eq:y1secondIT}) is written as
\begin{gather}
 (x-t a_1)(x- t a_2 ) y_1(qx) + \frac{\chi _1 \chi _2 t^2 a_1^2 a_2^2 (x-a_3) (x-q a_4 )}{q \theta _1^2 } y_1(x/q)\nonumber
\\ \qquad
{} - \bigg\{ \frac{t a_1 a_2 (\chi _2+ q \chi _1 )}{q \theta _1 } x^2 + E x + t^2 a_1 a_2 \bigg( 1+ \frac{\theta _2 }{\theta _1} \bigg) \bigg\} y_1 (x)=0,\label{eq:qHeuny1}
\end{gather}
and equation~(\ref{eq:cy1secondIT}) is written as
\begin{gather}
(x-t a_1 ) (x- t a_2 ) \check{y}_1(qx) + \frac{\chi _1 t a_1 a_2 }{q \theta _1 }\bigg( x -\frac{t \theta _2 }{\chi _1 a_4 } \bigg) \bigg( x- \frac{q t \theta _2 }{\chi _1 a_3 } \bigg) \check{y}_1(x/q) \nonumber
\\ \qquad
{} - \bigg\{ \frac{ q \chi _1 t a_1 a_2 + \theta _1}{q \theta _1 } x^2 + E x + t^2 a_1 a_2 \frac{\chi _2 t a_1 a_2 + \theta _2 }{\theta _1 } \bigg\} \check{y}_1 (x)=0.\label{eq:qHeuncy1}
\end{gather}
Note that these equations are the $q$-Heun equation, and the specialization of the parameters is related to the initial value space of $q$-$P\big(D^{(1)}_5\big)$ (see~\cite{STT1}).
By replacing the parameters, we obtain the following theorem on the $q$-Heun equation.

\begin{Theorem}
Assume that $l_1 l_2 l_3 l_4 = h_1 h_2 h_3 q^2$ and $ g( x) $ satisfies the $q$-Heun equation writ\-ten~as
\begin{gather}
\big(x-h_{1}q^{1/2}\big)\big(x-h_{2}q^{1/2}\big)g(x/q)+l_{3}l_{4}\big(x-l_{1}q^{-1/2}\big)\big(x-l_{2}q^{-1/2}\big)g(qx) \nonumber
\\ \qquad
{}-\big\{(l_{3}+l_{4})x^{2}+l_{3}l_{4}Ex+(l_{1}l_{2}l_{3}l_{4}h_{1}h_{2})^{1/2}\big(h_{3}^{1/2}+h_{3}^{-1/2}\big)\big\}g(x) =0.\label{eq:qHeungx}
\end{gather}
Let $\lambda $ be the value satisfying $ q^{\lambda} = q / l_4$.
Then the function
\begin{gather*}
\check{g}(x) = \int^{\xi \infty}_{0}\frac{g(s)}{s - x}P_{\lambda}(x, s)\, {\rm d}_{q}s
\end{gather*}
formally satisfies
\begin{gather}
\big(x-h'_{1}q^{1/2}\big)\big(x-h'_{2}q^{1/2}\big)\check{g}(x/q)+l'_{3}l'_{4}\big(x-l'_{1}q^{-1/2}\big)\big(x-l'_{2}q^{-1/2}\big)\check{g}(qx) \nonumber
 \\ \qquad
{}-\big\{(l'_{3}+l'_{4})x^{2}+l'_{3}l'_{4}Ex+(l'_{1}l'_{2}l'_{3}l'_{4}h'_{1}h'_{2})^{1/2}\big((h'_{3})^{1/2}+(h'_{3})^{-1/2}\big)\big\}\check{g}(x) =0,\label{eq:qHeuncgx}
\end{gather}
where
\begin{gather*}
l_1 '= l_{1},\quad\
l_2 '= l_{2}, \quad \
l_3 '=l_3,\quad\
l_4 '= q,\quad\
h_1 '= q h_{1} / l_4,\quad \
h_2 '= q h_{2} / l_4,\quad\
h_3 '= l_4 /(q h_{3} ).
\end{gather*}
\end{Theorem}

\begin{proof}
Let $l_1$, $l_2$, $l_3$, $l_4$, $h_1$, $h_2$ and $h_3$ be the values such that $l_1 l_2 l_3 l_4 = h_1 h_2 h_3 q^2 $.
We evaluate the values $a_1, a_2, a_3 $ and $a_4 $ as $a_3 = h_{1}q^{1/2}$, $a_4 = h_{2}q^{-1/2} $, $t a_1 = l_{1}q^{-1/2}$ and $t a_2 = l_{2}q^{-1/2}$, and fix the ratios $\theta _1 /\chi _1 $ and $\theta _1 /\chi _2 $ by $l_3 = \theta _1 / (\chi _1 t a_1 a_2 ) $ and $l_4 = q \theta _1 /( \chi _2 t a_1 a_2 ) $.
It follows from the relation $\chi_{1}\chi_{2} a_{1} a_{2} a_{3} a_{4} = \theta_{1}\theta_{2} $ in equation~(\ref{eq:relation1}) that $h_3 = l_1 l_2 l_3 l_4 /\big(h_1 h_2 q^2 \big) = \theta _1/ \theta _2$.
Then equation~(\ref{eq:qHeungx}) is written as equation~(\ref{eq:qHeuny1}) by setting $g(x)=y_1(x)$.
We apply Theorem \ref{thm:qmcint1}, where the parameters $y$ and $z$ are specialized as equation~(\ref{eq:yzqHeun}).
Then we obtain equation~(\ref{eq:qHeuncy1}) by the $q$-integral transformation.
By rewriting the parameters $a_1$, $a_2$, $a_3$, $a_4$, $\theta _1 /\chi _1$, $\theta _1 /\chi _2$ and~$\theta _1/ \theta _2$, we obtain equation~(\ref{eq:qHeuncgx}).
\end{proof}

\section{Concluding remarks} \label{sec:CR}
In this paper, we applied the $q$-middle convolution to a linear $q$-difference equation associated with the $q$-Painlev\'e VI equation, and we obtain integral transformations as a consequence.
We~investigated the symmetry by the $q$-middle convolution in terms of the affine Weyl group symmetry of the $q$-Painlev\'e VI equation.
As an application, we obtained an integral transformation on the $q$-Heun equation.
Note that our result is a $q$-analogue of the results in \cite{Fil,TakI,TakMH} on the Painlev\'e VI equation, the middle convolution and Heun's differential equation.

In \cite{SY}, the $q$-convolution was connected to the $q$-integral transformation by the Jackson integral.
On the other hand, the convolution for the system of Fuchsian differential equations was connected to the Euler's integral transformation, and there are several choice of the cycles on the integration by the Pochhammer contour.
Thus, there is a problem to find more cycles on the $q$-integral transformation associated with the $q$-convolution.
Other problems related with the general $q$-middle convolution may be found from our explicit application of the specified $q$-middle convolution.

The $q$-Painlev\'e VI equation was denoted by $q$-$P\big(D^{(1)}_5\big)$ in \cite{KNY}, and the Weyl group symmetry and the Lax pair of discrete Painlev\'e equations including $q$-$P\big(D^{(1)}_5\big)$ were reviewed in~\cite{KNY}.
The equations $q$-$P\big(E^{(1)}_6\big)$ and $q$-$P\big(E^{(1)}_7\big)$ are also $q$-analogue of the Painlev\'e VI equation.
We hope to extend the symmetry of integral transformations to the cases $q$-$P\big(E^{(1)}_6\big)$ and $q$-$P\big(E^{(1)}_7\big)$ by using the $q$-middle convolution, which might be related with the variants of $q$-Heun equation~\cite{TakR,TakqH}.
There is also a problem to connect the degenerated $q$-Painlev\'e equations \big(e.g.,~$q$-$P\big(A^{(1)}_4\big)$\big) with the $q$-middle convolution.
In this direction, the theory of the $q$-middle convolution for the non-Fuchsian $q$-difference equation is anticipated.

\appendix
\section{Middle convolution for other parameters}
In Section~\ref{sec:JSqMC}, we discussed the middle convolution which is related to the $q$-Painlev\'e VI equation.
The space $\mathcal{L}$ was defined in equation~(\ref{eq:KL}) and we imposed the condition $q^{\lambda}= \chi_{2}a_1a_2t / \theta_1 $ in Section~\ref{sec:JSqMC}, which induces $ \dim (\mathcal{L}) =1$.
In the appendix, we discuss the case $q^{\lambda}= \chi_{1}a_1a_2t / \theta_1 $, which also induces $ \dim (\mathcal{L}) =1$.

We continue the argument in the case $q^{\lambda}= \chi_{1}a_1a_2t / \theta_1 $ by replacing the matrix $P$ in equation~(\ref{eq:P}) with
\begin{gather*}
P = \begin{pmatrix}
 0 & 0 & 1 & v_{01} & 0 & 0 \\
 g_{1} & g_{3} & 0 & v_{02} & 0 & 0 \\
 0 & 0 & 1 & 0 & v_{11} & 0 \\
 g_{2} & g_{4} & 0 & 0 & v_{12} & 0 \\
 0 & 0 & 1 & 0 & 0 & v_{21} \\
 0 & 0 & 0 & 0 & 0 & v_{22} \\
 \end{pmatrix}\!,\qquad
\begin{array}{l}
 g_{3} = -q (\chi_{1} a_2 y -\theta_1 ) z+a_2 (y- t a_1),
 \\[.5ex]
 g_{4} = y (a_1-a_2) ( q \chi_{1} z-1),
\end{array}
\\[.5ex]
g_{1} = -\chi_{2} a_2 \big\{ q^2 \chi_{1} \theta_1 (y-a_3) (y-a_4) (\chi_{1} y-\chi_{2} t a_1 ) z^2 +\theta_1 (y-t a_1 )^2 (y- t a_2 )
\\ \hphantom{g_{1} = }
{} -q (y\!- t a_1 ) \big(2 \theta_1 \chi_{1} y^2\! -\theta_1 ( \chi_{2} t a_1\! + \chi_{1} t a_2\! +\chi_{1} a_3\! +\chi_{1} a_4) y\!+t \big(\theta_1 ^2\!+\chi_{1} \chi_{2} a_1 a_2 a_3 a_4\big)\big) z \big\},
 \\[.5ex]
 g_{2} = (a_1-a_2) y \big\{ q^2 \chi_{1} ^2 \chi_{2} \theta_1 (y-a_3) (y-a_4 ) z^2 +\theta_1 \chi_{2} (y- t a_1 ) (y- t a_2 )
\\ \hphantom{g_{2} = }
{}-q \chi_{1} \big(2 \chi_{2} \theta_1 y^2 - \chi_{2} \theta_1 (t a_1 +t a_2 +a_3+a_4) y+t \big(\theta_1 ^2+\chi_{2} ^2 a_1 a_2 a_3 a_4\big)\big) z \big\}.
\end{gather*}
Recall that $v_{01}, \dots, v_{22} $ were defined by equations~(\ref{eq:v01v02v11v12v21v22}) and (\ref{eq:v01v11v21}).
Then $\det P = q^3 w^2 y^3 z^3 \times\chi_{2}(\chi_{1} -\chi_{2} )^2 (a_1-a_2) \theta_1 ^2 ( \chi_{2} t a_1 a_2 -\theta_1 ) v_{22}^2$, and the matrix $P$ is invertible if $\det P \neq 0 $.
Set
\begin{gather*}
	\widetilde{F}_{1} = P^{-1}F_{1}P, \qquad
	\widetilde{F}_{2} = P^{-1}F_{2}P, \qquad
	\widetilde{F}_{\infty} = P^{-1}{F}_{\infty }P,
\end{gather*}
(see equation~(\ref{eq:F1F2Finf})).
Then they admit the following expression:
\begin{gather*}
\widetilde{F}_{1} = \begin{pmatrix}
\overline{F}_{1} & O \\ * & *
\end{pmatrix}\!, \qquad
\widetilde{F}_{2} = \begin{pmatrix}
\overline{F}_{2} & O \\ * & *
\end{pmatrix}\!, \qquad
\widetilde{F}_{\infty} = \begin{pmatrix}
\overline{F}_{\infty} & O \\ * & *
\end{pmatrix}\!,
\end{gather*}
where $\overline{F}_{1}$, $\overline{F}_{2}$, $\overline{F}_{\infty}$ are $2\times 2$ matrices given by
\begin{gather*}
\overline{F}_{\infty} = \begin{pmatrix}
1 & 0 \\ 0 & \chi_2 t a_1 a_2 /\theta_1
\end{pmatrix}\!,\qquad
\overline{F}_{2} = \overline{F}_{1} | _{a_1 \leftrightarrow a_2},
\\
\overline{F}_{1} = \frac{a_1}{ q y z \theta_1^2 (a_1-a_2) ( \chi_2 t a_1 a_2 - \theta_1 )}
\begin{pmatrix}
 -\theta_1 f_1^{[1]} f_2^{[1]} & - a_2 f_2^{[1]}
 \\[.5ex]
 \chi_{2} a_2 \theta_1 f_1^{[1]} f_3^{[1]} & \chi_{2} a_2^2 f_3^{[1]}
 \end{pmatrix}\!,
\end{gather*}
where
\begin{gather*}
 f_1^{[1]} = q \chi_{1} \chi_{2} a_2 (y-a_3) (y-a_4) z-( \chi_{2} a_2 y -\theta_1 ) (y- t a_2 ),
\\
f_2^{[1]} = q (\chi_{1} a_2 y -\theta_1 ) z-a_2 (y- t a_1 ),
\\
f_3^{[1]} = q^2 \chi_{1} \theta_1 (y-a_3) (y-a_4) (\chi_{1} y - \chi_{2} t a_1 ) z^2 +\theta_1 (y-t a_1 )^2 (y- t a_2 )
\\ \hphantom{f_3^{[1]} =}
{}-q (y\!- t a_1 ) \big(2 \chi_{1} \theta_1 y^2\!- \theta_1 ( \chi_{2} t a_1\! + \chi_{1} t a_2\! + \chi_{1} a_3\! + \chi_{1} a_4 ) y\! +t \big( \theta_1 ^2\!+ \chi_{1} \chi_{2} a_1 a_2 a_3 a_4 \big) \big) z.
\end{gather*}
Write
\begin{gather}
\overline{Y} (qx) = \overline{F}(x) \overline{Y} (x), \qquad
\overline{F}(x) = \overline{F}_{\infty} +
\frac{\overline{F}_{1}}{1 - x/(ta_{1})} + \frac{\overline{F}_{2}}{1 - x/(ta_{2})}. \label{eq:appqmc2}
\end{gather}
Note that the upper right entry of $\overline{F}(x) $ is written as
\begin{gather*}
\frac{t a_1 a_2 (( t a_1 a_2 -q \theta_1 z) x + t a_1 a_2 y (q \chi_{1} z-1)) }{q y z \theta_1 ^2 (x-t a_1 ) (x-t a_2 ) ( \chi_{2} t a_1 a_2 -\theta_1 ) }.
\end{gather*}
As discussed in Section~\ref{sec:JSqMC}, equation~(\ref{eq:appqmc2}) is related to an integral transformation.
Let $Y(x)$ be a solution to $Y(qx) = B(x)Y(x)$ in equation~(\ref{eq:JSlinDEqxB}) and write
$Y(x)= \left(\!\begin{smallmatrix}
y_{1}(x) \\ y_{2}(x)
\end{smallmatrix}\!\right)$.
By applying Theorem \ref{thm:qcint}, it is shown that the function $\overline{Y}(x) =
\left(\!\begin{smallmatrix}
\check{y}_{1}(x) \\ \check{y}_{2}(x)
\end{smallmatrix}\!\right)$
defined by
\begin{align*}
\check{y}_{j}(x) &= \int^{\xi \infty}_{0} \bigg\{ \bigg( \frac{p_{j1}}{s} + \frac{p_{j3}}{s - ta_{1}} + \frac{p_{j5}}{s - ta_{2}}\bigg)y_{1}(s)
+ \bigg( \frac{p_{j2}}{s} + \frac{p_{j4}}{s - ta_{1}} + \frac{p_{j6}}{s - ta_{2}} \bigg)y_{2}(s) \bigg\}
\\
& \hphantom{=}
\times P_{\lambda}(x, s)\, {\rm d}_{q}s, \qquad j=1,2,
\end{align*}
formally satisfies equation~(\ref{eq:appqmc2}), where $p_{jk}$ is the $(j,k)$-entry of the matrix $P^{-1}$.

By a straightforward calculation, it is shown that
\begin{gather*}
 \frac{p_{21}}{s} + \frac{p_{23}}{s - ta_{1}} + \frac{p_{25}}{s - ta_{2}} = \frac{-t \theta_{1}^2 (\chi_{1}-\chi_{2})}{c (\chi_{2} t a_1 a_2 -\theta_{1}) s} b_{21}(s),
 \\
 \frac{p_{22}}{s} + \frac{p_{24}}{s - ta_{1}} + \frac{p_{26}}{s - ta_{2}} = \frac{-t \theta_{1}^2 (\chi_{1}-\chi_{2})}{c (\chi_{2} t a_1 a_2 -\theta_{1}) s} (b_{22}(s) -1),
 \\
 \frac{p_{12}}{s} + \frac{p_{14}}{s - ta_{1}} + \frac{p_{16}}{s - ta_{2}} = \frac{\theta_{1} ((t a_1 a_2 - q z\theta_{1}) s + y t a_1 a_2 (q z\chi_{1} -1))}{q c w y z \chi_{2} a_1 a_2 (\chi_{2} t a_1 a_2 -\theta_{1}) (s-y) s} b_{12}(s),
 \\
 \frac{p_{11}}{s} + \frac{p_{13}}{s - ta_{1}} + \frac{p_{15}}{s - ta_{2}} = \frac{\theta_{1} ((t a_1 a_2 - q z\theta_{1}) s + y t a_1 a_2 (q z\chi_{1} -1))}{q c w y z \chi_{2} a_1 a_2 (\chi_{2} t a_1 a_2 -\theta_{1}) (s-y) s} b_{11}(s)
 \\ \hphantom{ \frac{p_{11}}{s} + \frac{p_{13}}{s - ta_{1}} + \frac{p_{15}}{s - ta_{2}} = }
{}- \frac{t (\chi_{1} (t a_1 a_2 -q z \theta_{1} )s +y \theta_{1} (q z \chi_{1} -1))}{q c w y z \chi_{2} (\chi_{2} t a_1 a_2 -\theta_{1}) (s-y) s},
 \\
 c= q^2 \chi_{1}^2 \chi_{2} \theta_{1} (y-a_3) (y-a_4) z^2 + \theta_{1} (y \chi_{2}-\chi_{1} t a_2 ) (y- t a_1 )
 \\ \hphantom{ c=}
{} -q \chi_{1} \big(2 \chi_{2} \theta_{1} y^2 -\theta_{1} (\chi_{2} t a_1 +\chi_{1} t a_2 +\chi_{2} a_3+\chi_{2} a_4) y + t \big(\chi_{1} \chi_{2} a_1 a_2 a_3 a_4+\theta_{1}^2\big)\big) z, \nonumber
\end{gather*}
where $b_{jk}(s) $ are elements of the matrix $B(s)$ in equation~(\ref{eq:Bx}).
It follows from $ y_{1}(qs) = b_{11}(s)y_{1}(s) + b_{12}(s)y_{2}(s) $ that
\begin{align*}
 \check{y}_{1}(x) &= \frac{(t a_1 a_2 - q z\theta_{1})}{q c w y z \chi_{2} a_1 a_2 (\chi_{2} t a_1 a_2 -\theta_{1}) }
 \\
 &\times\!\!\!\!\!\phantom{=}\int^{\xi \infty}_{0} \!\biggl\{\! - (\chi_{1} t a_1 a_2 \!-\!\theta_{1})y_1(s)
\!+\! \theta_{1} \bigg(\! 1\! + \!\frac{y t a_1 a_2 (q z\chi_{1}\! -\!1)}{(t a_1 a_2 \!-\! q z\theta_{1}) s} \bigg) ( y_1 (qs)\! -\!y(s))\! \biggr\} \frac{P_{\lambda}(x, s)}{s-y} \, {\rm d}_{q}s.
\end{align*}
Hence, the integral representation of $ \check{y}_{1}(x)$ in the case $q^{\lambda}= \chi_{1} t a_1a_2 / \theta_1 $ is more complicated than that in the case $q^{\lambda}= \chi_{2} t a_1a_2 / \theta_1 $.
On the other hand, it follows from $ y_{2}(qs) = b_{21}(s)y_{1}(s) + b_{22}(s)y_{2}(s) $ that
\begin{gather*}
\check{y}_{2}(x) = \frac{-t \theta_{1}^2 (\chi_{1}-\chi_{2})}{c (\chi_{2} t a_1 a_2 -\theta_{1}) } \int^{\xi \infty}_{0}\frac{y_{2}(qs) - y_{2}(s)}{s}P_{\lambda}(x, s)\, {\rm d}_{q}s.
\end{gather*}

To give the correspondence of the parameters by the $q$-middle convolution in the form of the equation $Y(qx)= \big\{ A_0 (t) + A_1 (t) x +A_2 x^2 \big\} Y(x) $ in equation~(\ref{eq:Ax012}), we need to transform equation~(\ref{eq:appqmcthm}).

Let $\widetilde{c}$ and $\widetilde{d}$ be a non-zero constant which will be fixed later.
Set $x= \widetilde{d} \widetilde{x}$,
\begin{gather*}
\widetilde{A}(\widetilde{x}) = \widetilde{c} (x- t a_1 ) (x- t a_2 ) \bigg( \overline{F}_{\infty} + \frac{\overline{F}_{1}}{1 - x/(ta_{1})} + \frac{\overline{F}_{2}}{1 - x/(ta_{2})} \bigg)
\end{gather*}
and write $\widetilde{A}(\widetilde{x}) = \widetilde{A}(\widetilde{x}, t) = \widetilde{A}_{0}(t) + \widetilde{A}_{1}(t) \widetilde{x} + \widetilde{A}_{2} \widetilde{x}^{2}$.
Then we have
\begin{gather*}
\widetilde{A}_{2} = \begin{pmatrix}
\widetilde{c} \widetilde{d}^2 \theta_{1} / (t a_{1}a_{2}) & 0 \\
0 & \widetilde{c} \widetilde{d}^2 \chi_{2}
\end{pmatrix}\!,
\\
\widetilde{A}_{0}(t) \mbox{ has the eigenvalues } \widetilde{c} \chi_{1} t^{2} a_{1} a_{2} \mbox{ and } \widetilde{c} t \theta_{2},
\\
\det \widetilde{A}(\widetilde{x}, t) =
\frac{\widetilde{c} ^{2} \widetilde{d}^4 \chi_{2} \theta_{1}}{ t a_{1}a_{2}} \bigg( \widetilde{x} - \frac{ta_{1}}{\widetilde{d}}\bigg) \bigg(\widetilde{x} - \frac{ta_{2}}{\widetilde{d}}\bigg)
\bigg(\widetilde{x} - \frac{\chi_{1} t a_{1}a_{2}a_{3}}{\widetilde{d}\theta_{1}}\bigg) \bigg(\widetilde{x} - \frac{\chi_{1} t a_{1}a_{2}a_{4}}{\widetilde{d}\theta_{1}}\bigg).
\end{gather*}
Hence the action of the $q$-middle convolution to the parameters in the case $q^{\lambda}= \chi_{1} t a_1a_2 / \theta_1 $ is described as
\begin{gather*}
\chi_{1}\rightarrow \frac{\widetilde{c} \widetilde{d}^2 \theta_{1} }{t a_{1}a_{2}}, \qquad
\chi_{2} \rightarrow \widetilde{c} \widetilde{d}^2 \chi_{2}, \qquad
\{a_{1}, a_{2}\} \rightarrow \bigg\{\frac{a_{1}}{\widetilde{d}}, \frac{a_{2}}{\widetilde{d}} \bigg\},
\\
\{a_{3}, a_{4}\} \rightarrow \bigg\{ \frac{\chi_{1} t a_{1}a_{2}a_{3}}{\widetilde{d} \theta_{1}}, \frac{\chi_{1} t a_{1}a_{2}a_{4}}{\widetilde{d} \theta_{1}} \bigg\}, \qquad
\{ t\theta_{1}, t\theta_{2} \} \rightarrow \big\{ \widetilde{c} \chi_{1} t^{2} a_{1} a_{2}, \widetilde{c} t \theta_{2} \big\}. \nonumber
\end{gather*}
We investigate the action to the parameters $y$ and $z$.
We denote the images of $y$ and $z$ by $\widetilde{y}$ and~$\widetilde{z}$.
Let $\widetilde{a}_{11} (\widetilde{x})$ (resp.~$\widetilde{a}_{12} (\widetilde{x})$) be the upper left entry (resp.~the upper right entry) of the matrix $\widetilde{A}(\widetilde{x})$.
Then the value $\widetilde{y}$ is the zero of the linear function $\widetilde{a}_{12} (\widetilde{x})$, and we have
\begin{gather*}
\widetilde{y} =
\frac{t a_1 a_2 (q \chi_{1} z-1)}{ \widetilde{d} (q \theta_1 z - t a_1 a_2 )}\, y
= \frac{\chi_{1} t a_1 a_2 (q z-1/\chi_{1})}{ \widetilde{d} \theta_1(q z - t a_1 a_2 /\theta_1 )} \, y.
\end{gather*}
The value $\widetilde{z}$ satisfies $\widetilde{a}_{11} (\widetilde{x}) | _{\widetilde{x}=\widetilde{y} } = \big(\widetilde{y} - t a_{1}/\widetilde{d}\big)\big(\widetilde{y} -t a_{2}/\widetilde{d}\big) /(q \widetilde{z}) $,
and we obtain
\begin{gather*}
\widetilde{z} = \frac{z }{\widetilde{c} \widetilde{d}^2}.
\end{gather*}
Set $\widetilde{d} =\chi_{1}t a_{1}a_{2} /\theta_{1}$ and $\widetilde{c} \widetilde{d}^2=1$.
By applying the correspondence in equation~(\ref{eq:JSKNYcorrespondence}), the transformation of the parameters is written as
\begin{gather*}
\nu_{1} \rightarrow \frac{\kappa _{2}}{\nu_{5}}, \qquad
\nu_{2} \rightarrow \nu_{2}, \qquad
\bigg\{ \frac{\kappa _{1}}{\nu_{7}}, \frac{\kappa _{1}}{\nu_{8}} \bigg\} \rightarrow \bigg\{ \frac{\kappa _{1}\kappa _{2}}{\nu_{7}\nu_{1}\nu_{5}}, \frac{\kappa _{1}\kappa _{2}}{\nu_{8}\nu_{1}\nu_{5}} \bigg\}, \qquad
\{\nu_{3},\nu_{4} \} \rightarrow \{\nu_{3},\nu_{4}\},
\\
\bigg\{\frac{\kappa _{2}}{\nu_{5}}, \frac{\kappa _{2}}{\nu_{6}} \bigg\} \rightarrow \bigg\{ \nu_{1}, \frac{\kappa _{2}}{\nu_{6}} \bigg\}, \qquad
f \rightarrow \frac{g - 1/\nu_1}{g - \nu_5/\kappa_2} \, f, \qquad
g \rightarrow g. \nonumber
\end{gather*}
Hence we obtain the following proposition.

\begin{Proposition}
We specify the parameters $\widetilde{c}$ and $\widetilde{d} $ by setting $\widetilde{d} =\chi_{1}t a_{1}a_{2} /\theta_{1}$ and $\widetilde{c} \widetilde{d}^2=1$.
Then the transformation of the parameters induced by the $q$-middle convolution in the case $q^{\lambda}= \chi_{1} t a_1a_2 / \theta_1 $ is realized by the action of $s_3$ given in equation~\eqref{eq:AffineWeylAction}.
\end{Proposition}

\subsection*{Acknowledgements}

The authors are grateful to the referees for careful reading of the manuscript and valuable comments.
The third author was supported by JSPS KAKENHI Grant Number JP18K03378.

\pdfbookmark[1]{References}{ref}
\LastPageEnding


\begin{thebibliography}{99}
\footnotesize\itemsep=0pt

\bibitem{Aom}
Aomoto K., On the 3 fundamental problems concerning $q$-basic hypergeometric
 functions ($q$-difference equations, asymptotic behaviours and connection
 problem), in Proceedinds of Fifth Oka Symposium (March 18--19, 2006, Nara),
 Oka Mathematical Institute, Japan, 2006, 16~pages (in Japanese), available at
 \url{http://www.nara-wu.ac.jp/omi/oka_symposium/05/aomoto.pdf}.

\bibitem{AT}
Arai Y., Takemura K., On $q$-middle convolution and $q$-hypergeometric
 equations, {i}n preparation.

\bibitem{DR1}
Dettweiler M., Reiter S., An algorithm of {K}atz and its application to the
 inverse {G}alois problem, \href{https://doi.org/10.1006/jsco.2000.0382}{\textit{J.~Symbolic Comput.}} \textbf{30} (2000),
 761--798.

\bibitem{DR2}
Dettweiler M., Reiter S., Middle convolution of {F}uchsian systems and the
 construction of rigid differential systems, \href{https://doi.org/10.1016/j.jalgebra.2007.08.029}{\textit{J.~Algebra}} \textbf{318}
 (2007), 1--24.

\bibitem{Fil}
Filipuk G., On the middle convolution and birational symmetries of the sixth
 {P}ainlev\'e equation, \textit{Kumamoto~J. Math.} \textbf{19} (2006), 15--23.

\bibitem{JS}
Jimbo M., Sakai H., A {$q$}-analog of the sixth {P}ainlev\'e equation,
 \href{https://doi.org/10.1007/BF00398316}{\textit{Lett. Math. Phys.}} \textbf{38} (1996), 145--154,
 \href{https://arxiv.org/abs/chao-dyn/9507010}{arXiv:chao-dyn/9507010}.

\bibitem{KNY}
Kajiwara K., Noumi M., Yamada Y., Geometric aspects of {P}ainlev\'e equations,
 \href{https://doi.org/10.1088/1751-8121/50/7/073001}{\textit{J.~Phys.~A: Math. Theor.}} \textbf{50} (2017), 073001, 164~pages,
 \href{https://arxiv.org/abs/1509.08186}{arXiv:1509.08186}.

\bibitem{Katz}
Katz N.M., Rigid local systems, \textit{Annals of Mathematics Studies}, Vol.~139, \href{https://doi.org/10.1515/9781400882595}{Princeton University Press}, Princeton, NJ, 1996.

\bibitem{Sak}
Sakai H., Rational surfaces associated with affine root systems and geometry of
 the {P}ainlev\'e equations, \href{https://doi.org/10.1007/s002200100446}{\textit{Comm. Math. Phys.}} \textbf{220} (2001),
 165--229.

\bibitem{SY}
Sakai H., Yamaguchi M., Spectral types of linear {$q$}-difference equations and
 {$q$}-analog of middle convolution, \href{https://doi.org/10.1093/imrn/rnw089}{\textit{Int. Math. Res. Not.}}
 \textbf{2017} (2017), 1975--2013, \href{https://arxiv.org/abs/1410.3674}{arXiv:1410.3674}.

\bibitem{STT1}
Sasaki S., Takagi S., Takemura K., $q$-Heun equation and initial-value space of
 $q$-Painlev\'e equation, in Pro\-ceedings of the Conference
 FASnet21, {t}o appear, \href{https://arxiv.org/abs/2110.13860}{arXiv:2110.13860}.

\bibitem{Tkg}
Takagi S., Application of $q$-middle convolution to $q$-sixth Painlev\'e
 equation, {M}aster Thesis, Chuo University, 2021 (in Japanese).

\bibitem{TakI}
Takemura K., Integral representation of solutions to {F}uchsian system and
 {H}eun's equation, \href{https://doi.org/10.1016/j.jmaa.2007.11.015}{\textit{J.~Math. Anal. Appl.}} \textbf{342} (2008), 52--69,
 \href{https://arxiv.org/abs/0705.3358}{arXiv:0705.3358}.

\bibitem{TakMH}
Takemura K., Middle convolution and {H}eun's equation, \href{https://doi.org/10.3842/SIGMA.2009.040}{\textit{SIGMA}}
 \textbf{5} (2009), 040, 22~pages, \href{https://arxiv.org/abs/0810.3112}{arXiv:0810.3112}.

\bibitem{TakR}
Takemura K., Degenerations of {R}uijsenaars--van {D}iejen operator and
 {$q$}-{P}ainlev\'e equations, \href{https://doi.org/10.1093/integr/xyx008}{\textit{J.~Integrable Syst.}} \textbf{2} (2017),
 xyx008, 27~pages, \href{https://arxiv.org/abs/1608.07265}{arXiv:1608.07265}.

\bibitem{TakqH}
Takemura K., On {$q$}-deformations of the {H}eun equation, \href{https://doi.org/10.3842/SIGMA.2018.061}{\textit{SIGMA}}
 \textbf{14} (2018), 061, 16~pages, \href{https://arxiv.org/abs/1712.09564}{arXiv:1712.09564}.

\end{thebibliography}
\end{document}